\documentclass[11pt,reqno]{article}
\usepackage{bbding}
\usepackage{pifont}
\usepackage{amsmath}
\usepackage{mathrsfs}
\usepackage{amsfonts,bm}
\usepackage{amsthm}
\usepackage{epsfig}
\usepackage[]{harpoon}
\usepackage[active]{srcltx}
\usepackage{indentfirst,latexsym}
\usepackage{psfrag}
\usepackage[]{amssymb}
\usepackage{cite}
\usepackage{color}

\newcommand{\new}{\newcommand*}\new{\rnew}{\renewcommand*}
\new{\newe}{\newenvironment*}\new{\stl}{\setlength}
\stl{\textwidth}{155mm}\stl{\textheight}{22cm}\stl{\headheight}{0cm}
\stl{\topmargin}{0cm}\stl{\oddsidemargin}{0.5cm}\stl{\evensidemargin}{0cm}
\rnew{\arraystretch}{1.2}\rnew{\baselinestretch}{1.2}
\renewcommand{\thefootnote}{\ding{73}}
\newtheorem{thm}{Theorem}[section]

\newtheorem{lem}{Lemma}[section]
\newtheorem{prop}{Proposition}
\newtheorem{defn}{Definition}[section]
\newtheorem{rem}{Remark}[section]
\newtheorem{prob}{Problem}[section]

\newcommand{\eps}{\varepsilon}

\newcommand{\fr}{\frac}

\newcommand{\pa}{\partial}

\numberwithin{equation}{section}

\newe{keywords}
   {\begin{quote}{\bf Keywords:}}
      {\end{quote}}
\newe{AMS}
   {\begin{quote}{\bf AMS subject classification 2000:}}
      {\end{quote}}
\newe{MSC}{\vspace*{5mm}
    {\noindent\it Mathematics Subject Classification(2000):}}{}
\new{\sect}[1]{\section{#1}\setcounter{equation}{0}
 \setcounter{thm}{0}\setcounter{lmm}{0}\setcounter{rmk}{0} }

\begin{document}

\title{ Sonic-supersonic solutions for the two-dimensional steady full Euler equations  }

\author{
Yanbo Hu$^{a,*}$, Jiequan Li$^{b,c}$
\\{\small \it $^a$Department of Mathematics, Hangzhou Normal University,
Hangzhou, 311121, China}
\\
{\small \it $^b$Laboratory of Computational Physics, Institute of Applied Physics}
\\
{\small \it and Computational Mathematics, Beijing, 100088, China}
\\
{\small \it $^c$Center for Applied Physics and Technology, Peking University, 100871, China}}

\rnew{\thefootnote}{\fnsymbol{footnote}}

\footnotetext{ $^*$Corresponding author. }
\footnotetext{ Email address: yanbo.hu@hotmail.com (Y. Hu), li\_jiequan@iapcm.ac.cn (J. Li). }

\date{May 2017}

\maketitle
\begin{abstract}
This paper focuses on the structure of classical  sonic-supersonic solutions near sonic curves for the two-dimensional full Euler equations in gas dynamics.
In order to deal with the parabolic degeneracy  near the sonic curve,  a novel set of dependent and independent variables are introduced  to transform the Euler equations into a new system of governing equations which displays a clear regularity-singularity structure. With the help of technical characteristic decompositions, the existence of a local smooth solution  for the new system is first  established in a weighted metric space by using the iteration method and then expressed in terms of the original physical variables.   This is the first time to construct
a classical sonic-supersonic solution near a sonic curve for the full Euler equations.
\end{abstract}

\begin{keywords}
Two-dimensional full Euler equations, sonic-supersonic solution, sonic curve, partial hodograph transformation, weighted metric space.
\end{keywords}

\begin{AMS}
35L65, 35L80, 76H05.
\end{AMS}

\section{Introduction}\label{S1}

In the famous book  (Supersonic Flow and Shock Waves, 1948), Courant and Friedrichs described the following transonic phenomena in a duct: {\em Suppose the flow through an infinite duct whose walls are plane except for a small inward bulge at some section,
if the entrance proper Mach number is not much below the value one, then the flow becomes supersonic in a finite region adjacent to the bulge.}  See Figure 1 for illustration. A similar situation ubiquitously occurs in the context of gas dynamics, such as  a flow over an airfoil.
\vspace{0.2cm}

\begin{figure}[htbp]
\begin{center}
\includegraphics[scale=0.65]{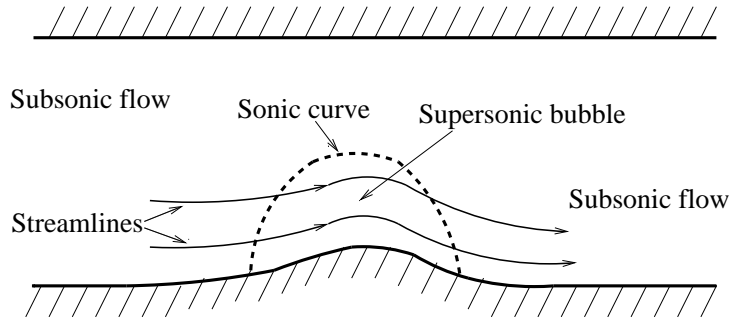}
\caption{\footnotesize Transonic phenomena in a duct.}
\end{center}
\end{figure}

The existence of solution for such a  transonic flow problem remains open  mathematically for a long time and many contributions have been made, as reviewed below. The two-dimensional steady isentropic irrotational compressible Euler equations are usually adopted to study this problem for mathematical simplicity. Under the irrotational assumption, the isentropic steady Euler equations are transformed into a two-by-two quasilinear reducible equations or a single second order nonlinear (potential) equation. In \cite{Morawetz1}, Morawetz showed the nonexistence of smooth solutions for the problem in general. The existence of weak solutions were investigated by Morawetz \cite{Morawetz2} and Chen et al. \cite{Chen1, Chen2, Chen3} in  the  compensated-compactness framework. Xie and Xin  adopted a potential-stream function formulation to verify the existence of solutions in a subsonic-sonic part of the nozzle \cite{Xie-Xin1, Xie-Xin2}. The study of transonic shocks arising in supersonic flow past a blunt body or a bounded nozzle was presented   in \cite{ChenS2, Elling-Liu, Xin-Yin1, Xin-Yin2}.
For more related references, refer to,  for example,  works on classical methods for solutions \cite{Bers, Courant1, Ferrari, Guderley}, on perturbation arguments and linear theory \cite{BG, ChenS1, Oleinik, Smirnov} and on asymptotic models \cite{Canic, Cole-Cook, Hunter}. One may also consult the books \cite{Courant1, Courant2, Kuzmin} for some explicit examples of transonic solutions, e.g., the Ringleb flow. However, it is quite difficult to apply them for practical applications.
\vspace{0.2cm}

A number of results for studying the transonic flow problem are based on the hodograph method which allows to  linearize the two-by-two equations by switching the roles of the velocity with the spatial independent variables. However, it is well-known that this method is difficult in taking on boundary conditions and in transforming back to the original independent variables due to the sonic degeneracy. Instead, some  attempts have been made  to work directly in the spatial independent variables or other coordinates. For example, Kuz'min \cite{Kuzmin} used the coordinate system $(\Psi, \Phi)$ to establish an existence theorem for the problem of the perturbation of a given known transonic solutions, where $\Psi$ is the stream function and $\Phi$ is the potential function. Recently, Zhang and Zheng \cite{ZZ1}  constructed a local smooth supersonic solution on one side of a given sonic curve for the two-dimensional steady isentropic irrotational Euler system and shed more insights into clear structures of the solution near the sonic curve. Zhang and Zheng's result is based on the coordinate system $(\sqrt{q^2-c^2},\Phi)$, where $q$ is the flow speed and $c$ is the sound speed. We also refer the reader to \cite{Hu-Li-Sheng, Hu-Wang, Lai-Sheng, Li, Li-Zheng1, Li-Zheng2, Lim, SWZ, Song-Zheng, Wang-Zheng} and references therein for the relevant results about the degenerate Goursat problems arising from the study of two-dimensional Riemann problem to the Euler equations and its related models (such as pressure-gradient equations).
\vspace{0.2cm}

This paper aims to construct a classical sonic-supersonic solution to the two-dimensional steady full Euler equations. Since there may contain transonic shocks in the transonic flow \cite{Morawetz1} and the entropy in the flow is not constant and the flow behind the shock is not irrotational \cite{Majda}, thus it is more suitable to adopt the steady full Euler equations as governing equations to study the transonic flow problem,
\begin{align}\label{1.1}
\left\{
\begin{array}{l}
  (\rho u)_x+(\rho v)_y=0, \\
  (\rho u^2+p)_x+(\rho uv)_y=0, \\
  (\rho uv)_x+(\rho v^2+p)_y=0, \\
  (\rho Eu+pu)_x+(\rho Ev+pv)_y=0,
\end{array}
\right.
\end{align}
where $\rho$, $(u, v)$, $p$ and $E$ are, respectively, the density, the velocity, the pressure and the specific total energy. For  polytropic gases,  $E=\fr{u^2+v^2}{2}+\frac{1}{\gamma-1}\frac{p}{\rho}$,  where $\gamma>1$ is the adiabatic gas constant. This system has three eigenvalues
\begin{align}\label{1.2}
\Lambda_0=\frac{v}{u},\ \ \Lambda_{\pm}=\frac{uv\pm c\sqrt{u^2+v^2-c^2}}{u^2-c^2},
\end{align}
where $c=\sqrt{\gamma p/\rho}$ is the speed of sound.  Therefore, it is of mixed-type: supersonic for $q>c$, subsonic for $q<c$ and sonic for $q=c$. The notation $q=\sqrt{u^2+v^2}$ is used for the flow speed.
The set of points at which $c=q$ is called the {\em sonic curve.} We consider the following problem.
\begin{prob}\label{p1}
Given a piece of smooth curve $\Gamma: y=\varphi(x)$,
$x\in [x_1,x_2]$, we assign the boundary data for $(\rho, u, v, p)$ on $\Gamma$, $(\rho, u, v, p)(x,\varphi(x)) =(\hat{\rho}, \hat{u}, \hat{v}, \hat{p})(x)$ such that $\hat{\rho}(x)>0$, $\hat{p}(x)>0$ and $\gamma\hat{p}(x)=\hat{\rho}(x)(\hat{u}^2(x)+\hat{v}^2(x))$ for any $x\in[x_1,x_2]$. This means that $\Gamma$ is a sonic curve.
We want to find a classical solution for system \eqref{1.1} in the region $q>c$ near $\Gamma$.
\end{prob}
\vspace{0.2cm}

The main result we obtain is stated in the following theorem.
\begin{thm}\label{thm2}
Let $\hat{\theta}$ be the flow angle on $\Gamma$ defined by $\hat{\theta}=\arctan(\hat{v}/\hat{u})(x)$. Assume that the curve $\Gamma$ and the boundary data $(\hat{\rho}, \hat{u}, \hat{v}, \hat{p})$ satisfy
\begin{align}\label{1.3}
\begin{array}{c}
\varphi(x)\in C^4([x_1,x_2]),\quad (\hat{\rho}, \hat{u}, \hat{v}, \hat{p})(x)\in C^4([x_1,x_2]),  \\
\hat{p}'\leq0,\quad \hat{\theta}'<0,\quad \cos\hat{\theta}\varphi'-\sin\hat{\theta}>0,\quad \cos\hat{\theta}+\sin\hat{\theta}\varphi'>0, \quad \forall\ x\in[x_1,x_2].
\end{array}
\end{align}
Then there exists a classical solution for Problem \ref{p1} in the region $q>c$ near $\Gamma$.
\end{thm}

The approach in this paper is inspired the method by Zhang and Zheng  in \cite{ZZ1, ZZ2} to deal with the steady isentropic irrotational Euler equations and pressure-gradient equations. However, compared to the isentropic and irrotational case, there are two main difficulties  arising from the effect of entropy and vorticity. First, we need to seek an appropriate coordinate system which  features the parabolic  degeneracy of the problem under consideration near the sonic curve. In contrast with the isentropic irrotational case, the non-existence of potential function $\Phi$ for the full Euler flows makes it impossible to use the coordinate system $(\sqrt{q^2-c^2},\Phi)$,  which  plays a key role in Zhang and Zheng's work,  for  the current  problem.   To overcome this difficulty, a new set of variables including the Mach angle function $\omega$ and the flow angle function $\theta$ are introduced, and  $(\cos\omega, \theta)$ are chosen  as the  independent coordinate system in order to clearly characterize the degeneracy near the sonic curve. It is observed that the partial hodograph transformation $(x,y)\rightarrow(\cos\omega, \theta)$ reduces the Euler equations to a new system which has a clearer regularity-singularity structure.  See \eqref{3.9} or \eqref{4.2} below. Second, we  define some new functions as the dependent variables which  makes  the  new system more convenient to analyze. In Zhang and Zheng's previous works,  they adopted  the characteristic decomposition in terms of local sound speed $c$ \cite{Li-Z-Z} and chose $(\bar{\pa}^+c/c,\bar{\pa}^-c/c)$ as the dependent quantities  in order to write the isentropic irrotational Euler equations into a tidy first-order hyperbolic system. However, for the full Euler equations, the characteristic decomposition in terms of  $c$ is considerably formidable  so that  $(\bar{\pa}^+c/c,\bar{\pa}^-c/c)$ are not suitable  to be  the dependent quantities any more. Instead, two novel variables $H$ and $\Xi$,  functions of Mach angle $\omega$, entropy $S$ and Bernoulli number $B$,  are introduced to derive the characteristic decomposition in terms of  $\Xi$.  See \eqref{a2} and \eqref{2.14} in Section 2. The variable $H$, together with $\theta$ and $\omega$, constitutes a closed subsystem.  The characteristic decomposition in terms of  $\Xi$ has a pretty symmetrical form. Then  $(\bar{\pa}^+\Xi,\bar{\pa}^-\Xi, H)$  are taken as the dependent quantities to obtain a closed system in the partial hodograph $(\cos\omega, \theta)$-plane. It turns out that this approach works well for the current problem and the related semi-hyperbolic problems \cite{Hu-Li}.
\vspace{0.2cm}

In order to give readers a better understanding of the procedure of the paper, we use the approach to deal with the classical Tricomi equation in Appendix \ref{sa}. We point out that the same problem for the Tricomi equation was studied by many investigators employing the fundamental solution method, e.g., \cite{Bi, Han-Hong}, while it cannot be applied to the nonlinear equations and also the transonic flow problems.
\vspace{0.2cm}

The rest of the paper is organized as follows.
In Section \ref{s2}, we introduce a new set of dependent variables, including the inclination angles and the variable $\Xi$, in order to write the Euler equations \eqref{1.1} in the characteristic form, and present the corresponding result based on  characteristic  decompositions.
In Section \ref{s3}, we reformulate the  problem by introducing a novel pair of independent variables and using the characteristic decomposition for $\Xi$. In Section \ref{s4}, we employ the iteration method to establish the existence of local classical solutions for the newly reformulated problem in a weighted metric space. We complete the proof of the main theorem by returning the  classical solution in the partial hodograph plane  into that in terms of original physical variables in Section \ref{s5}. Finally, the appendices are provided for the procedure to the classical Tricomi equation and the technique arguments of characteristic decompositions.

\section{Basic characteristic  decompositions in terms of Mach and flow angles}\label{s2}

In order to analyze the nonlinear problem under consideration in this paper, it is convenient to use the Mach angles, the flow angles, the entropy and the Bernoulli quantity as dependent variables. Hence we rewrite the governing equations, provide basic characteristic decompositions and reformulate the problem in this new framework.

\subsection{Full Euler equations and characteristic decompositions}

Assume that the solution of \eqref{1.1} is  smooth. Then the system can be rewritten as
\begin{align}\label{2.1}
\textbf{A}\textbf{W}_x+\textbf{B}\textbf{W}_y=0,
\end{align}
where the primitive variables and the coefficient matrices are
\begin{align*}
\textbf{W}=\left(
 \begin{array}{c}
    \rho \\
    u \\
    v \\
    p
 \end{array}
\right), \quad
\textbf{A}=\left(
 \begin{array}{cccc}
    u & \rho & 0 & 0 \\
    0 & u & 0 & \frac{1}{\rho} \\
    0 & 0 & u & 0 \\
    0 & \gamma p & 0 & u
 \end{array}
\right),
\quad \textbf{B}=
\left(
 \begin{array}{cccc}
    v & 0 & \rho & 0 \\
    0 & v & 0 & 0 \\
    0 & 0 & v & \frac{1}{\rho} \\
    0 & 0 & \gamma p & v
 \end{array}
\right).
\end{align*}
The eigenvalues $\Lambda$ are defined by finding the roots of $\|\Lambda \textbf{A} -\textbf{B}\|=0$, as expressed in \eqref{1.2}.
The left eigenvectors associated with the eigenvalues $\Lambda_0,\ \Lambda_{\pm}$ are, respectively,
$$
\ell_{01}=(0,u,v,0),\quad \ell_{02}=(c^2,0,0,-1),\quad \ell_\pm=(0,-\Lambda_\pm\gamma p,\gamma p,\Lambda_\pm u-v).
$$
Thus system \eqref{2.1} can be turned into the characteristic form, by standard manipulation,
\begin{align}\label{2.2}
\left\{
\begin{array}{l}
  uS_x+vS_y=0, \\
  uB_x+vB_y=0, \\
  -c\rho vu_x+c\rho uv_x\pm\sqrt{u^2+v^2-c^2}p_x \\
  \ \ \ \ +\Lambda_\pm(-c\rho vu_y+c\rho uv_y\pm\sqrt{u^2+v^2-c^2}p_y)=0,
\end{array}
\right.
\end{align}
where $S=p\rho^{-\gamma}$ is the entropy function and $B=\fr{u^2+v^2}{2}+\fr{c^2}{\gamma-1}$ is the Bernoulli function.
\vspace{0.2cm}

We introduce the inclination angles of characteristics as follows
\begin{align}\label{2.3}
\tan\alpha=\Lambda_+,\quad \tan\beta=\Lambda_-,\quad \tan\theta=\Lambda_0,
\end{align}
from which we have
\begin{align}\label{2.4}
\theta=\fr{\alpha+\beta}{2},\quad u=c\fr{\cos\theta}{\sin\omega},\quad v=c\fr{\sin\theta}{\sin\omega},
\end{align}
where $\omega=\fr{\alpha-\beta}{2}$ is the Mach angle function. Moreover, we introduce the following normalized directional derivatives along the characteristics
\begin{align}\label{2.5}
\begin{array}{l}
\bar{\pa}^+=\cos\alpha\pa_x+\sin\alpha\pa_y,\quad
\bar{\pa}^-=\cos\beta\pa_x+\sin\beta\pa_y, \\
\bar{\pa}^0=\cos\theta\pa_x+\sin\theta\pa_y, \ \quad \bar{\pa}^\perp=-\sin\theta\pa_x+\cos\theta\pa_y.
\end{array}
\end{align}
Then we have
\begin{align}\label{2.6}
\left\{
\begin{array}{l}
\pa_x=-\fr{\sin\beta\bar{\pa}^+-\sin\alpha\bar{\pa}^-}{\sin(2\omega)},\\[5pt] \pa_y=\fr{\cos\beta\bar{\pa}^+-\cos\alpha\bar{\pa}^-}{\sin(2\omega)},
\end{array}
\right. \qquad
\left\{
\begin{array}{l}
\bar{\pa}^0=\fr{\bar{\pa}^++\bar{\pa}^-}{2\cos\omega}, \\[5pt]
\bar{\pa}^\perp=\fr{\bar{\pa}^+-\bar{\pa}^-}{2\sin\omega}.
\end{array}
\right.
\end{align}
In terms of the variables $(S, B, \omega, \theta)$, system \eqref{2.2} can be transformed into a new form
\begin{align}\label{2.7}
\left\{
\begin{array}{l}
   \bar{\pa}^0S=0, \\
   \bar{\pa}^0B=0, \\
   \bar{\pa}^+\theta+\fr{\cos^2\omega}{\sin^2\omega+\kappa}\bar{\pa}^+\omega=
\fr{\sin(2\omega)}{4\kappa}\bigg(\fr{1}{\gamma}\bar{\pa}^+\ln S-\bar{\pa}^+\ln B\bigg), \\
   \bar{\pa}^-\theta-\fr{\cos^2\omega}{\sin^2\omega+\kappa}\bar{\pa}^-\omega=-
\fr{\sin(2\omega)}{4\kappa}\bigg(\fr{1}{\gamma}\bar{\pa}^-\ln S-\bar{\pa}^-\ln B\bigg),
  \end{array}
\right.
\end{align}
or
\begin{align}\label{a4}
\left\{
\begin{array}{l}
   \bar{\partial}^0S=0, \\
   \bar{\partial}^0B=0, \\
   \bar{\partial}^0\theta+\frac{\cos\omega\sin\omega}{\sin^2\omega+\kappa}\bar{\partial}^\perp\omega=
\fr{\sin^2\omega}{2\kappa} \bigg(\frac{1}{\gamma}\bar{\partial}^\perp\ln S-\bar{\partial}^\perp\ln B\bigg), \\
   \bar{\partial}^\perp\theta+\frac{\cos^2\omega\cot\omega}{\sin^2\omega+\kappa}\bar{\partial}^0\omega=0,
  \end{array}
\right.
\end{align}
where $\kappa=(\gamma-1)/2$. The detailed derivation of \eqref{2.7} is presented in Appendix \ref{App1}.  From this system, it is obvious that $(\theta, \omega, S, B)$ are more suitable to be taken as dependent variables.  Moreover, we want to further write this system in terms of characteristic directions $\bar\partial^+$ and $\bar\partial^-$ because the characteristic decompositions are more easily taken.  Before this,  we state the commutator relation between $\bar{\pa}^0$ and $\bar{\pa}^+$. The justification is  given  in Appendix \ref{App2}. Similar commutator relation between $\bar{\pa}^-$ and $\bar{\pa}^+$ can be found in \cite{Li-Z-Z}.

\begin{prop}
For any smooth quantity $I(x,y)$, there holds
\begin{align}\label{a7}
\bar{\partial}^0\bar{\partial}^+I-\bar{\partial}^+\bar{\partial}^0I=
\frac{1}{\sin\omega}[(\cos\omega\bar{\partial}^+\theta-
\bar{\partial}^0\alpha)\bar{\partial}^0I-
(\bar{\partial}^+\theta-\cos\omega\bar{\partial}^0\alpha)\bar{\partial}^+I].
\end{align}
\end{prop}
For any smooth function $I$ satisfying $\bar{\pa}^0I\equiv0$, we obtain by \eqref{2.6} that $\bar{\pa}^+I=-\bar{\pa}^-I$. Moreover, we use \eqref{a7} to find that
\begin{align}\label{2.8}
\bar{\partial}^0\bar{\partial}^+I=-\frac{\bar{\partial}^+\theta-
\cos\omega\bar{\partial}^0\alpha}{\sin\omega}\bar{\partial}^+I.
\end{align}
We employ \eqref{2.6} and the last two equations of \eqref{2.7} to obtain
\begin{align*}
\bar{\partial}^+\theta-\cos\omega\bar{\partial}^0\alpha
&=\frac{\bar{\partial}^+\theta-\bar{\partial}^-\theta}{2}-
\frac{\bar{\partial}^+\omega+\bar{\partial}^-\omega}{2}   \\
&=-\frac{(\kappa+1)\cos\omega}{\kappa+\sin^2\omega}\bar{\partial}^0\omega.
\end{align*}
Inserting this into \eqref{2.8} gives
\begin{align*}
\bar{\partial}^0\bar{\partial}^+I=
\frac{(\kappa+1)\cot\omega}{\kappa+\sin^2\omega}\bar{\partial}^0\omega\bar{\partial}^+I,
\end{align*}
and
\begin{align}\label{2.9}
\bar{\partial}^0\bar{\partial}^+I=\fr{\bar{\partial}^0G\cdot\bar{\partial}^+I}{G},
\end{align}
where
\begin{align}\label{2.10}
G=G(\omega)=\bigg(\fr{\sin^2\omega}{\kappa+\sin^2\omega}\bigg)^{\fr{\kappa+1}{2\kappa}}.
\end{align}
Thus we deduce from \eqref{2.9}
\begin{align}\label{2.11}
\bar{\partial}^0\bigg(\fr{\bar{\partial}^+I}{G}\bigg)=0.
\end{align}

Thanks to the first two equations of \eqref{2.7}, one has $\bar{\pa}^0\ln S=0$ and $\bar{\pa}^0\ln B=0$ which imply that $\bar{\pa}^+\ln S=-\bar{\pa}^-\ln S$ and $\bar{\pa}^+\ln B=-\bar{\pa}^-\ln B$. Therefore we have
$$
\bar{\pa}^0\bigg(\fr{1}{4\kappa\gamma}\ln S-\fr{1}{4\kappa}\ln B\bigg)=0, \ \ \bar{\pa}^+\bigg(\fr{1}{4\kappa\gamma}\ln S-\fr{1}{4\kappa}\ln B\bigg)
=-\bar{\pa}^-\bigg(\fr{1}{4\kappa\gamma}\ln S-\fr{1}{4\kappa}\ln B\bigg).
$$
Denote
\begin{align}\label{a2}
H=\fr{\bar{\pa}^+\bigg(\fr{1}{4\kappa\gamma}\ln S-\fr{1}{4\kappa}\ln B\bigg)}{G} =-\fr{\bar{\pa}^-\bigg(\fr{1}{4\kappa\gamma}\ln S-\fr{1}{4\kappa}\ln B\bigg)}{G}.
\end{align}
Then we use \eqref{2.7} and \eqref{2.11} to derive a system in terms of  the dependent variables $(\theta, \omega, H, S)$
\begin{align}\label{2.12}
\left\{
\begin{array}{l}
   \bar{\pa}^+\theta+\fr{\cos^2\omega}{\sin^2\omega+\kappa}\bar{\pa}^+\omega=
\sin(2\omega)GH, \\
   \bar{\pa}^-\theta-\fr{\cos^2\omega}{\sin^2\omega+\kappa}\bar{\pa}^-\omega=
\sin(2\omega)GH,\\
\bar{\pa}^0H=0,\\
\bar{\pa}^0S=0.
  \end{array}
\right.
\end{align}

\begin{rem}\label{rem1}
System \eqref{2.12} is equivalent to system \eqref{2.7} for smooth solutions. In addition, the first three equations of \eqref{2.12} constitute a closed subsystem.
\end{rem}


We further introduce a new variable
\begin{align}\label{2.14}
\Xi=\fr{1}{4\kappa}\ln\bigg(\fr{\sin^2\omega}{\kappa+\sin^2\omega}\bigg)-\fr{1}{4\kappa}\bigg(\fr{1}{\gamma}\ln S-\ln B\bigg),
\end{align}
and employ the last two equations of \eqref{2.7} to arrive at
\begin{align}\label{2.15}
\left\{
    \begin{array}{l}
       \bar{\pa}^+\theta+\sin(2\omega)\bar{\pa}^+\Xi=0, \\
       \bar{\pa}^-\theta-\sin(2\omega)\bar{\pa}^-\Xi=0,
    \end{array}
\right.
\end{align}
or
\begin{align}\label{a1}
\left\{
    \begin{array}{l}
       \bar{\pa}^0\theta+2\sin^2\omega\bar{\pa}^\perp \Xi=0, \\
       \bar{\pa}^\perp\theta+2\cos^2\omega\bar{\pa}^0\Xi=0.
    \end{array}
\right.
\end{align}
The introduction of the new variable $\Xi$ allows us to obtain the following characteristic decompositions.
\begin{align}\label{2.16}
\left\{
\begin{array}{l}
\bar{\pa}^-\bar{\pa}^+\Xi =\fr{\kappa\bar{\pa}^+\Xi+(\kappa+\sin^2\omega)GH}{\cos^2\omega}[\bar{\pa}^+\Xi -\cos(2\omega)\bar{\pa}^-\Xi]+\fr{\bar{\pa}^+\Xi}{\cos^2\omega}[\bar{\pa}^+\Xi+\cos^2(2\omega)\bar{\pa}^-\Xi], \\[6pt]
\bar{\pa}^+\bar{\pa}^-\Xi =\fr{\kappa\bar{\pa}^-\Xi-(\kappa+\sin^2\omega)GH}{\cos^2\omega}[\bar{\pa}^-\Xi -\cos(2\omega)\bar{\pa}^+\Xi]+\fr{\bar{\pa}^-\Xi}{\cos^2\omega}[\bar{\pa}^-\Xi+\cos^2(2\omega)\bar{\pa}^+\Xi].
\end{array}
\right.
\end{align}
The derivation of the characteristic decompositions in terms of  $\Xi$ is  given in Appendix \ref{App3}.

\subsection{The boundary data and restatement of the main result}

We  consider the boundary data for system \eqref{2.12} subject to the boundary condition \eqref{1.3}. Denote
\begin{align}\label{a6}
\hat{S}(x)=\hat{p}(x)\hat{\rho}^{-\gamma}(x)>0,\quad \hat{B}(x)=\fr{\hat{u}^2(x)+\hat{v}^2(x)}{2}+\fr{\gamma}{\gamma-1}\fr{\hat{p}(x)}{\hat{\rho}(x)}>0,\ \ \forall\ x\in[x_1,x_2].
\end{align}
These are boundary data for the entropy function $S$ and the Bernoulli quantity.
 Note that $\omega={\pi}/{2}$ on $\Gamma$. Then by the definition of $H$, one has the boundary data of $H$ on $\Gamma$
\begin{align}\label{2.17}
\hat{H}(x)=\fr{1}{2\gamma(\gamma-1)}\bigg(\fr{\gamma+1}{2}\bigg)^{\fr{\gamma+1}{2(\gamma-1)}}\cdot \fr{\hat{B}^\gamma}{\hat{S}}\cdot\bar{\pa}^+\bigg(\fr{S}{B^\gamma}\bigg)\bigg|_{\Gamma}.
\end{align}
Making use of \eqref{a6} and noticing the fact $c=q$ on $\Gamma$ yields
\begin{align}\label{a8}
\fr{\hat{B}^\gamma}{\hat{S}} =\fr{\bigg(\fr{\gamma\hat{p}}{2\hat{\rho}}+\fr{\gamma}{\gamma-1}\fr{\hat{p}}{\hat{\rho}}\bigg)^\gamma} {\hat{p}\hat{\rho}^{-\gamma}} =\bigg(\fr{\gamma(\gamma+1)}{2(\gamma-1)}\bigg)^\gamma\hat{p}^{\gamma-1}.
\end{align}
Moreover, we employ the equation $\bar{\pa}^0S=0$ and the notation $\hat{S}(x)=S(x,\varphi(x))$ on $\Gamma$ to obtain
$$
S_x(x,\varphi(x))=\fr{\sin\hat{\theta}\hat{S}'}{\sin\hat{\theta}-\cos\hat{\theta}\varphi'},\quad S_y(x,\varphi(x))=\fr{\cos\hat{\theta}\hat{S}'}{\cos\hat{\theta}\varphi'-\sin\hat{\theta}},
$$
and
$$
\bar{\pa}^+S|_{\Gamma}=\fr{\hat{S}'}{\cos\hat{\theta}\varphi'-\sin\hat{\theta}}.
$$
Similarly, one has
$$
\bar{\pa}^+B|_{\Gamma}=\fr{\hat{B}'}{\cos\hat{\theta}\varphi'-\sin\hat{\theta}}.
$$
Hence there holds
$$
\bar{\pa}^+\bigg(\fr{S}{B^\gamma}\bigg)\bigg|_{\Gamma} =\fr{1}{\cos\hat{\theta}\varphi'-\sin\hat{\theta}}\bigg(\fr{\hat{S}}{\hat{B}^\gamma}\bigg)',
$$
which, combined with \eqref{a8}, gives
$$
\bar{\pa}^+\bigg(\fr{S}{B^\gamma}\bigg)\bigg|_{\Gamma} =\bigg(\fr{2(\gamma-1)}{\gamma(\gamma+1)}\bigg)^\gamma \fr{(1-\gamma)\hat{p}^{-\gamma}\hat{p}'}{\cos\hat{\theta}\varphi'-\sin\hat{\theta}}.
$$
Putting the above and \eqref{a8} into \eqref{2.17}, one  arrives at
$$
\hat{H}(x)=-\bigg(\fr{\gamma+1}{2}\bigg)^{\fr{\gamma+1}{2(\gamma-1)}} \fr{\hat{p}'(x)}{2\gamma\hat{p}(x)[\cos\hat{\theta}(x)\varphi'(x)-\sin\hat{\theta}(x)]}.
$$

Therefore we obtain the boundary data $(\theta, \omega, H, S)$ on $\Gamma$ with
\begin{align}\label{2.18}
\theta=\hat{\theta}(x)\in C^4([x_1,x_2]),\quad \omega=\fr{\pi}{2},\quad H=\hat{H}(x)\in C^3([x_1,x_2]),\ \ \ \hat S(x) \in C^4([x_1,x_2]).
\end{align}
The constraints \eqref{1.3} on the boundary data become
\begin{align}\label{2.19}
\hat{H}\geq0, \quad \hat{\theta}'<0,\quad
\cos\hat{\theta}\varphi'-\sin\hat{\theta}>0,  \quad
\cos\hat{\theta}+\sin\hat{\theta}\varphi'>0, \ \ \forall\ x\in[x_1,x_2].
\end{align}
Then Theorem \ref{thm2} is restated    in the next theorem.
\begin{thm}\label{thm1}
Let conditions \eqref{2.19} be satisfied. Then the boundary problem \eqref{2.12} --\eqref{2.18} admits a
classical solution in the region $\omega<\fr{\pi}{2}$ near the sonic curve $\Gamma$.
\end{thm}

\section{Reformulated problem in a partial hodograph plane}\label{s3}

Since the parabolic degeneracy on the sonic curve may result in singularity, we need to single out  the feature of governing equations near the sonic curve. For this purpose, we introduce a new partial hodograph transformation and derive a new system of governing equations.

\subsection{A partial hodograph transformation}

Denote $(\widetilde{U}, \widetilde{V})=(\bar{\pa}^+\Xi, \bar{\pa}^-\Xi)$. Then we have  a new system in terms of the variables $(\widetilde{U}, \widetilde{V}, H)$ from \eqref{2.12} and \eqref{2.16},
\begin{align}\label{3.1}
\left\{
\begin{array}{l}
\bar{\pa}^-\widetilde{U} =\fr{\kappa\widetilde{U}+(\kappa+\sin^2\omega)GH}{\cos^2\omega}[\widetilde{U} -\cos(2\omega)\widetilde{V}]+\fr{\widetilde{U}}{\cos^2\omega}[\widetilde{U}+\cos^2(2\omega)\widetilde{V}], \\[4pt]
\bar{\pa}^+\widetilde{V} =\fr{\kappa\widetilde{V}-(\kappa+\sin^2\omega)GH}{\cos^2\omega}[\widetilde{V} -\cos(2\omega)\widetilde{U}]+\fr{\widetilde{V}}{\cos^2\omega}[\widetilde{V}+\cos^2(2\omega)\widetilde{U}], \\
\bar{\pa}^0H=0.
\end{array}
\right.
\end{align}
The boundary data for  $\widetilde{U}$ and $\widetilde{V}$ on $\Gamma$ are prescribed as follows.  Due to \eqref{2.6}, one has $\bar{\pa}^+\Xi+\bar{\pa}^-\Xi=2\cos\omega\bar{\pa}^0\Xi$,  which implies that $\bar{\pa}^+\Xi=-\bar{\pa}^-\Xi$ on $\Gamma$. Making use of \eqref{a1} leads to $\bar{\pa}^+\Xi=-\bar{\pa}^-\Xi=-\bar{\pa}^0\theta/2$ on $\Gamma$.
By using \eqref{a1} again, we see that $\bar{\pa}^\perp\theta=0$ on $\Gamma$, which, together with the boundary value $\theta=\hat{\theta}(x)$,  gives
\begin{align*}
\theta_x(x,\varphi(x))=\fr{\hat{\theta}'\cos\hat{\theta}}{\cos\hat{\theta}+\varphi'\sin\hat{\theta}}, \quad \theta_y(x,\varphi(x))=\fr{\hat{\theta}'\sin\hat{\theta}}{\cos\hat{\theta}+\varphi'\sin\hat{\theta}},
\end{align*}
and
\begin{align}\label{3.2}
\widetilde{U}|_\Gamma=-\widetilde{V}|_\Gamma =-\fr{\bar{\pa}^0\theta}{2}\bigg|_\Gamma=\fr{-\hat{\theta}'}{2(\cos\hat{\theta}+\varphi'\sin\hat{\theta})}. =:-\widetilde{a}_0(x)>0
\end{align}
where  \eqref{2.19} is used.
\vspace{0.2cm}

Now we introduce a partial hodograph transformation $(x,y)\rightarrow(t,r)$ by defining
\begin{align}\label{3.3}
t=\cos\omega(x,y),\quad r=\theta(x,y).
\end{align}
The Jacobian of this transformation is
\begin{align}\label{3.4}
J:=\fr{\pa(t,r)}{\pa(x,y)}&=\sin\omega(\theta_x\omega_y-\theta_y\omega_x) \nonumber \\
&=\fr{\bar{\pa}^+\omega\bar{\pa}^-\theta-\bar{\pa}^+\theta\bar{\pa}^-\omega}{2\cos\omega}
\nonumber \\
&=\sin\omega(\bar{\pa}^+\omega\bar{\pa}^-\Xi+\bar{\pa}^-\omega\bar{\pa}^+\Xi).
\end{align}
Thanks to \eqref{2.14} and \eqref{a2}, one has
\begin{align}\label{a3}
\bar{\pa}^\pm\omega=\fr{2\sin\omega(\kappa+\sin^2\omega)}{\cos\omega}(\bar{\pa}^\pm\Xi\pm GH).
\end{align}
Putting the above into \eqref{3.4} suggests
\begin{align}\label{3.5}
J=\fr{2F}{t}[2\widetilde{U}\widetilde{V}+GH(\widetilde{V}-\widetilde{U})],
\end{align}
where
\begin{align}\label{3.6}
F=F(t)=(1-t^2)(\kappa+1-t^2),\quad G=G(t)=\bigg(\fr{1-t^2}{\kappa+1-t^2}\bigg)^{\fr{\kappa+1}{2\kappa}}.
\end{align}
We combine \eqref{3.5} and \eqref{3.2} and recall the condition $\hat{H}\geq0$ to see that $J\neq0$ away from $t=0$. Clearly, the singularity near the sonic curve $t=0$ is singled out.
\vspace{0.2cm}

In terms of this new coordinates $(t,r)$, one has
$$
\bar{\pa}^i=-\sin\omega\bar{\pa}^i\omega\pa_t+\bar{\pa}^i\theta\pa_r, \ \ i=\pm,0,
$$
or
\begin{align}\label{3.8}
\begin{array}{l}
\bar{\pa}^+=-\fr{2F}{t}(\widetilde{U}+GH)\pa_t-2\sqrt{1-t^2}\widetilde{U}t\pa_r, \\[4pt]
\bar{\pa}^-=-\fr{2F}{t}(\widetilde{V}-GH)\pa_t+2\sqrt{1-t^2}\widetilde{V}t\pa_r, \\[4pt]
\bar{\pa}^0=-\fr{F}{t^2}(\widetilde{U}+\widetilde{V})\pa_t +\sqrt{1-t^2}(\widetilde{V}-\widetilde{U})\pa_r.
\end{array}
\end{align}
Denote $\overline{U}(t,r)=\widetilde{U}(x(t,r),y(t,r)), \overline{V}(t,r)=\widetilde{V}(x(t,r),y(t,r)), \overline{H}(t,r)=H(x(t,r),y(t,r))$. Then we obtain a new closed system for the variables $(\overline{U}, \overline{V}, \overline{H})$ under the coordinates $(t,r)$ as follows:
\begin{align}\label{3.9}
\left\{
\begin{array}{l}
\overline{U}_t-\fr{\sqrt{1-t^2}t^2\overline{V}}{F(\overline{V}-G\overline{H})}\overline{U}_r =-\fr{(\kappa+1)\overline{U}+(\kappa+1-t^2)G\overline{H}}{2F(\overline{V}-G\overline{H})}\cdot \fr{\overline{U}+\overline{V}}{t} +\fr{(\kappa+2-2t^2)\overline{U}+(\kappa+1-t^2)G\overline{H}}{F(\overline{V}-G\overline{H})}\overline{V}t, \\[8pt]
\overline{V}_t+\fr{\sqrt{1-t^2}t^2\overline{U}}{F(\overline{U}+G\overline{H})}\overline{V}_r =-\fr{(\kappa+1)\overline{V}-(\kappa+1-t^2)G\overline{H}}{2F(\overline{U}+G\overline{H})}\cdot \fr{\overline{U}+\overline{V}}{t} +\fr{(\kappa+2-2t^2)\overline{V}-(\kappa+1-t^2)G\overline{H}}{F(\overline{U}+G\overline{H})}\overline{U}t, \\[8pt]
\overline{H}_t+\fr{\sqrt{1-t^2}(\overline{U}-\overline{V})t^2}{F(\overline{U}+\overline{V})}\overline{H}_r=0.
\end{array}
\right.
\end{align}
We comment that system \eqref{3.9} is not a continuously differentiable system since it contains a singular factor $(\overline{U}+\overline{V})/t$.

\subsection{Boundary data in the partial hodograph plane}

We note that the sonic curve $\Gamma$: $y=\varphi(x)$, $x\in[x_1,x_2]$ on the $(x,y)$ plane is transformed to a segment on $t=0$ with $r\in[r_1,r_2]$ on the $(t,r)$ plane. Indeed, due to the assumption $\hat{\theta}'(x)<0$ by \eqref{2.19},  $r=\hat{\theta}(x)$ is a strictly decreasing smooth function,  which implies that  it can be expressed as  $x=\overline{x}(r)$ for $r\in[r_1,r_2]$.
\vspace{0.2cm}

We derive the value $\bar{\pa}^0\Xi$ on $\Gamma$. It follows from the fourth equation of \eqref{a4} and the second equation of \eqref{a1} that
\begin{align}\label{3.10}
\bar{\pa}^0\Xi=\fr{\bar{\pa}^0\sin\omega}{2\sin\omega(\kappa+\sin^2\omega)}.
\end{align}
Using the third equation of \eqref{a4}, \eqref{2.6} and \eqref{a2}, one obtains
$$
\bar{\pa}^\perp\sin\omega=-\fr{\kappa+\sin^2\omega}{\sin\omega}\bar{\pa}^0\theta+2(\kappa+\sin^2\omega)GH,
$$
which, along with \eqref{2.5} and \eqref{3.2}, leads to
\begin{align}\label{3.11}
-\sin\hat{\theta}(\pa_x\sin\omega)|_{\Gamma}+\cos\hat{\theta}(\pa_y\sin\omega)|_{\Gamma} &=-(\kappa+1)(\bar{\pa}^0\theta)_\Gamma+2G_0\hat{H} \nonumber \\
&=: 2(\kappa+1)(-\widetilde{a}_0+G_0\hat{H}),
\end{align}
where $G_0=(\kappa+1)^{-\fr{\kappa+1}{2\kappa}}$. Recalling the fact $\sin\omega=1$ on $\Gamma$ gives
$$
(\pa_x\sin\omega)|_{\Gamma}+\varphi'(\pa_y\sin\omega)|_{\Gamma}=0,
$$
which combined with \eqref{3.11} yields
$$
(\pa_x\sin\omega)|_{\Gamma}=\fr{-2(\kappa+1)\varphi'(-\widetilde{a}_0+G_0\hat{H})}{\cos\hat{\theta} +\varphi'\sin\hat{\theta}},\quad (\pa_y\sin\omega)|_{\Gamma}=\fr{2(\kappa+1)(-\widetilde{a}_0+G_0\hat{H})}{\cos\hat{\theta} +\varphi'\sin\hat{\theta}}.
$$
Putting the above into \eqref{3.10}, we get by \eqref{2.19} and \eqref{3.2}
\begin{align}\label{3.12}
(\bar{\pa}^0\Xi)_\Gamma=\fr{(\bar{\pa}^0\sin\omega)_\Gamma}{2(\kappa+1)} =\fr{\sin\hat{\theta}-\varphi'\cos\hat{\theta}}{\cos\hat{\theta} +\varphi'\sin\hat{\theta}}(-\widetilde{a}_0+G_0\hat{H})=:\widetilde{a}_1(x)<0.
\end{align}

Let $a_0(r)=\widetilde{a}_0(\overline{x}(r))$, $a_1(r)=\widetilde{a}_1(\overline{x}(r))$ and $H_0(r)=\hat{H}(\overline{x}(r))$. We study system \eqref{3.9} with the following boundary conditions:
\begin{align}
&\overline{U}(0,r)=-a_0(r),\quad \overline{V}(0,r)=a_0(r),\quad \overline{H}(0,r)=H_0(r), \label{3.13} \\
&\overline{U}_t(0,r)=a_1(r),\quad \overline{V}_t(0,r)=a_1(r),\quad \overline{H}_t(0,r)=0, \label{3.14}
\end{align}
for $r\in[r_1,r_2]$. The conditions in \eqref{3.13} are obvious, while the conditions in \eqref{3.14} come from system \eqref{3.9}, \eqref{3.12} and the requirement for continuous differentiability of solutions. It is not difficult to verified that
\begin{align}\label{3.15}
\begin{array}{c}
(a_0(r), a_1(r), H_0(r))\in C^3([r_1,r_2]), \\
H_0(r)\geq0 \ \ a_0(r)\leq-\eps_0,\ \  a_1(r)\leq-\eps_0
\end{array}
\end{align}
for some constant positive $\eps_0$. So we reformulate Problem \ref{p1} into the following new problem in the hodograph plane.

\begin{prob}\label{prob2}
Assume \eqref{3.15} holds. We want to seek a local classical solution for system \eqref{3.9} with the boundary conditions \eqref{3.13}-\eqref{3.14} in the region $t>0$.
\end{prob}

\section{The existence theorem in the hodograph plane}\label{s4}

This section serves to establish the existence of  smooth solution locally near the sonic curve in the hodograph plane.  Recall the boundary data in \eqref{3.13} and \eqref{3.14}.  We make the Taylor expansion for $(\overline{U}, \overline{V}, \overline{H})$ and introduce the higher order error terms for the variables $(\overline{U}, \overline{V}, \overline{H})$ as follows,
\begin{align}\label{4.1}
\left\{
\begin{array}{l}
U=\overline{U}+a_0(r)-a_1(r)t, \\
V=\overline{V}-a_0(r)-a_1(r)t, \\
W=\overline{H}-H_0(r).
\end{array}
\right.
\end{align}
Then system \eqref{3.9} can be transformed to
\begin{align}\label{4.2}
\left\{
\begin{array}{l}
U_t-\fr{\sqrt{1-t^2}(V+a_0+a_1t)t^2}{F(t)[V-G(t)W+\psi(t,r)]}U_r=\fr{U+V}{2t}+b_1(U,V,W,t,r), \\[5pt]
V_t+\fr{\sqrt{1-t^2}(U-a_0+a_1t)t^2}{F(t)[U+G(t)W+\phi(t,r)]}V_r=\fr{U+V}{2t}+b_2(U,V,W,t,r), \\[5pt]
W_t+\fr{\sqrt{1-t^2}(U-V-2a_0)t^2}{F(t)(U+V+2a_1t)}W_r=b_3(U,V,W,t,r),
\end{array}
\right.
\end{align}
where $F(t)$ and $G(t)$ are defined as in \eqref{3.6}, $\psi(t,r)=a_0(r)+a_1(r)t-G(t)H_0(r)$, $\phi(t,r)=-a_0(r)+a_1(r)t+G(t)H_0(r)$, and
\begin{align*}
b_1(U,V,W,t,r)&=-\bigg(\fr{U+V}{2t}+a_1\bigg)\bigg\{\fr{(\kappa+1)(U+V+2a_1t)}{F[V-GW+\psi]} \nonumber \\[4pt] &\qquad \qquad -\fr{t^2[(\kappa+2-t^2)(V+a_0+a_1t) -(\kappa+1-t^2)G(W+H_0)]}{F[V-GW+\psi]}\bigg\}
 \nonumber \\[4pt]
 &+\fr{t^2\sqrt{1-t^2}(-a_{0}'+a_{1}'t)(V+a_0+a_1t)}{F[V-GW+\psi]}
 \nonumber \\[4pt]
 &+\fr{(\kappa+2-2t^2)(U-a_0+a_1t)+(\kappa+1-t^2)G(W+H_0)}{F[V-GW+\psi]}(V+a_0+a_1t)t,
\end{align*}
\begin{align*}
b_2(U,V,W,t,r)=&-\bigg(\fr{U+V}{2t}+a_1\bigg)\bigg\{\fr{(\kappa+1)(U+V+2a_1t)}{F[U+GW+\phi]} \nonumber \\[4pt] &\qquad \qquad -\fr{t^2[(\kappa+2-t^2)(U-a_0+a_1t)+(\kappa+1-t^2)G(W+H_0)]}{F[U+GW+\phi]}\bigg\}
 \nonumber \\[4pt]
 &-\fr{t^2\sqrt{1-t^2}(a_{0}'+a_{1}'t)(U-a_0+a_1t)}{F[U+GW+\phi]}
 \nonumber \\[4pt]
 &+\fr{(\kappa+2-2t^2)(V+a_0+a_1t)-(\kappa+1-t^2)G(W+H_0)}{F[U+GW+\phi]}(U-a_0+a_1t)t,
\end{align*}
\begin{align*}
b_3(U,V,W,t,r)=-\fr{t^2\sqrt{1-t^2}(U-V-2a_0)H_{0}'}{F(U+V+2a_1t)}.
\end{align*}
The three eigenvalues of system \eqref{4.2} are expressed as
\begin{align}\label{4.4}
\lambda_1(t,r)&=-\fr{\sqrt{1-t^2}(V+a_0+a_1t)t^2}{F(t)[V-G(t)W+\psi(t,r)]},\
\lambda_2(t,r)=\fr{\sqrt{1-t^2}(U-a_0+a_1t)t^2}{F(t)[U+G(t)W+\phi(t,r)]}, \nonumber \\[5pt]
\lambda_3(t,r)&=\fr{\sqrt{1-t^2}(U-V-2a_0)t^2}{F(t)(U+V+2a_1t)}.
\end{align}

Corresponding to the boundary conditions \eqref{3.13}-\eqref{3.14}, one has
\begin{align}\label{4.3}
\begin{array}{l}
U(0,r)=V(0,r)=W(0,r)=0, \\
U_t(0,r)=V_t(0,r)=W_t(0,r)=0,
\end{array}
\quad r\in[r_1,r_2].
\end{align}
We now define a region in the plane $(t,r)$
$$
D_\delta=\{(t,r)|\ 0\leq t\leq\delta,\ \bar{r}_1(t)\leq r\leq \bar{r}_2(t)\},
$$
where $\bar{r}_1(t)$ and $\bar{r}_2(t)$ are smooth functions satisfying $\bar{r}_1(0)=r_1$, $\bar{r}_2(0)=r_2$ and $\bar{r}_1(t)<\bar{r}_2(t)$ for $t\in[0,\delta]$. Hence, Problem \ref{prob2} is equivalent to the following problem.

\begin{prob}\label{prob3}
Assume \eqref{3.15} holds. Then we want to seek a classical solution for system \eqref{4.2} with the boundary condition \eqref{4.3} in the region $D_\delta$ for some constant $\delta>0$.
\end{prob}

\subsection{A weighted metric  space}

Following Zhang and Zheng \cite{ZZ1}, we first give the definition of admissible functions and strong determinate domain to system \eqref{4.2}.
\begin{defn}[Admissible functions]\label{def1}
The vector function $\mathbf{F}=(f_1(t,r),f_2(t,r),f_3(t,r))^T$, defined for all $(t,r)\in D_\delta$, is an  admissible vector function if the following hold:

\noindent (i) The functions $f_i(i=1,2,3)$ are continuous on the region $D_\delta$.

\noindent (ii) The functions $f_i(i=1,2,3)$ satisfy the boundary value conditions in \eqref{4.3}.

\noindent (iii) There holds $\big\|\fr{f_1}{t^2}\big\|_\infty+\big\|\fr{f_2}{t^2}\big\|_\infty +\big\|\fr{f_3}{t^2}\big\|_\infty\leq M$ for some positive constant $M$.
\end{defn}
We denote $\mathcal{W}_{\delta}^M$  the set of all admissible vector functions. For an admissible vector function $\mathbf{F}=(f_1,f_2,f_3)^T\in\mathcal{W}_{\delta}^M$, we define the characteristic curves $r_i(t;\xi,\eta) (i=1,2,3)$ passing through a point $(\xi,\eta)\in D_\delta$ as follows:
\begin{align}\label{4.5}
\left\{
\begin{array}{l}
\fr{\rm d}{{\rm d}t}r_i(t;\xi,\eta)=\lambda_i(t, r_i(t;\xi,\eta)), \\
r_i(\xi;\xi,\eta)=\eta,
\end{array}
\right.
\end{align}
where $\lambda_i (i=1,2,3)$ are defined in \eqref{4.4}.

\begin{defn}[Strong determinate domains]\label{def2}
We call $D_\delta$ a strong determinate domain for system \eqref{4.2} if for any admissible vector function $\mathbf{F}=(f_1,f_2,f_3)^T$ and for any point $(\xi,\eta)\in D_\delta$, the characteristic curves $r_i(t;\xi,\eta) (i=1,2,3)$ stay insider $D_\delta$ for all $0\leq t\leq\xi$ until the intersection with the line $t=0$.
\end{defn}

We proceed to  construct solutions for Problem \ref{prob3} in a function class $\mathcal{S}_{\delta}^M$ which incorporates all continuously differentiable vector functions $\mathbf{F}=(f_1,f_2,f_3)^T: D_\delta\rightarrow\mathbb{R}^3$ satisfying the following properties:
\begin{align*}
&{\rm (P1)}\ \mathbf{F}(0,r)=\mathbf{F}_t(0,r)=0, \rule{10cm}{0ex} \\[3pt]
&{\rm (P2)}\ \bigg\|\fr{\mathbf{F}(t,r)}{t^2}\bigg\|_\infty\leq M, \rule{10cm}{0ex}
\\[3pt]
&{\rm (P3)}\ \bigg\|\fr{\pa_r\mathbf{F}(t,r)}{t^2}\bigg\|_\infty\leq M, \rule{10cm}{0ex}
\\[3pt]
&{\rm (P4)}\ \pa_r\mathbf{F}(t,r)\ {\rm is\ Lipschitz\ continuous\ with\ respect\ to\ r\ and}\ \bigg\|\fr{\pa_{rr}\mathbf{F}(t,r)}{t^2}\bigg\|_\infty\leq M,
\end{align*}
where $\|\cdot\|_\infty$ denotes the supremum norm on the domain $D_\delta$. We note that $\mathcal{S}_{\delta}^M$ is a subset of $\mathcal{W}_{\delta}^M$ and both $\mathcal{S}_{\delta}^M$ and $\mathcal{W}_{\delta}^M$ are subsets of $C^0(D_\delta, \mathbb{R}^3)$. For any elements $\mathbf{F}=(f_1,f_2,f_3)^T, \mathbf{G}=(g_1,g_2,g_3)^T$ in the set $\mathcal{W}_{\delta}^M$, we define the weighted metric as follows:
\begin{align}\label{4.6}
d(\mathbf{F}, \mathbf{G}):=\bigg\|\fr{f_1-g_1}{t^2}\bigg\|_\infty+ \bigg\|\fr{f_2-g_2}{t^2}\bigg\|_\infty +\bigg\|\fr{f_3-g_3}{t^2}\bigg\|_\infty.
\end{align}
It is not difficult to check that $(\mathcal{W}_{\delta}^M,d)$ is a complete metric space, while the subset $(\mathcal{S}_{\delta}^M,d)$ is not closed in the space $(\mathcal{W}_{\delta}^M,d)$.
\vspace{0.2cm}

For Problem \ref{prob3}, we have the the following existence theorem, to be proved in the next subsection.
\begin{thm}\label{thm3}

Let conditions \eqref{3.15} be fulfilled and $D_{\delta_0}$ be a strong determinate domain for
system \eqref{4.2}. Then there exists positive constants $\delta\in(0,\delta_0)$ and $M$ such that the boundary problem \eqref{4.2}--\eqref{4.3} admits a classical solution in the function class $\mathcal{S}_{\delta}^M$.
\end{thm}

\subsection{The proof of Theorem \ref{thm3}}

We establish Theorem \ref{thm3} by using the fixed point method. The proof is divided into four steps. In Step 1, we employ system \eqref{4.2} to construct an integration iteration mapping in the function class $\mathcal{S}_{\delta}^M$. In Step 2, we establish a series of {\em a priori} estimates for $b_i$ and $\lambda_i \ (i=1,2,3)$. In Step 3, we use the above estimates to demonstrate the mapping is a contraction, which implies that the iteration sequence converge to a  vector function in the limit. Finally, in Step 4, we show that this limit vector function also belongs to $\mathcal{S}_{\delta}^M$.
\vspace{0.2cm}

\noindent {\bf Step 1 (The iteration mapping).} Let vector function $(u,v,w)^T(t,r)\in \mathcal{S}_{\delta}^M$.
Denote
\begin{align}\label{4.7}
\fr{{\rm d}}{{\rm d}_1t}:=\pa_t+\lambda_{1}(t,r)\pa_r, \quad  \fr{{\rm d}}{{\rm d}_2t}:=\pa_t+\lambda_{2}(t,r)\pa_r, \quad \fr{{\rm d}}{{\rm d}_3t}:=\pa_t+\lambda_{3}(t,r)\pa_r,
\end{align}
where $\lambda_i (i=1,2,3)$ are defined as in \eqref{4.4} but with $u,v,w$ replacing $U,V,W$ respectively. Then we consider the equations
\begin{align}\label{4.8}
\left\{
\begin{array}{l}
\fr{{\rm d}}{{\rm d}_1t}U=\fr{u+v}{2t}+b_1(u,v,w,t,r), \\
\fr{{\rm d}}{{\rm d}_2t}V=\fr{u+v}{2t}+b_2(u,v,w,t,r), \\
\fr{{\rm d}}{{\rm d}_3t}W=b_3(u,v,t,r).
\end{array}
\right.
\end{align}
The integral form of \eqref{4.8} is
\begin{align}\label{4.9}
\left\{
\begin{array}{l}
U(\xi,\eta)=\displaystyle\int_{0}^\xi \bigg(\fr{u+v}{2t}+\widetilde{b}_1\bigg)(t,r_1(t;\xi,\eta))\ {\rm d}t, \\[10pt]
V(\xi,\eta)=\displaystyle\int_{0}^\xi \bigg(\fr{u+v}{2t}+\widetilde{b}_2\bigg)(t,r_2(t;\xi,\eta))\ {\rm d}t, \\[10pt]
W(\xi,\eta)=\displaystyle\int_{0}^\xi \widetilde{b}_3(t,r_3(t;\xi,\eta))\ {\rm d}t,
\end{array}
\right.
\end{align}
where $r_i(t;\xi,\eta)$ are defined as in \eqref{4.5} and
$$
\widetilde{b}_i(t,r_i(t;\xi,\eta)) =b_i(u(t,r_i(t;\xi,\eta)),v(t,r_i(t;\xi,\eta)),w(t,r_i(t;\xi,\eta)),t,r_i(t;\xi,\eta))
$$
for $i=1,2,3$.
We note that \eqref{4.9} determines an iteration mapping $\mathcal{T}$:
\begin{align}\label{4.10}
\mathcal{T}
\left(
\begin{array}{l}
u \\
v \\
w
\end{array}
\right)=
\left(
\begin{array}{l}
U \\
V \\
W
\end{array}
\right).
\end{align}
It is clear that the existence of classical solutions for the boundary problem \eqref{4.2} \eqref{4.3} is equivalent to the existence of fixed point for the mapping $\mathcal{T}$ in the function class $\mathcal{S}_{\delta}^M$.
\vspace{0.2cm}

\noindent {\bf Step 2 (A priori estimates).}  For further convenience, we hereinafter derive a series of estimates about $b_i$ and $\lambda_i$ (i=1,2,3). We use $K>1$ to denote a constant depending only on $\eps_0$, $\kappa$, the bounds of $F, G$ and the $C^3$ norms of $a_0,a_1,H_0$, which may change from one expression to another. Since $(u,v,w)^T\in \mathcal{S}_{\delta}^M$, by \eqref{3.15}, there exists a small constant $\delta_0<1$ such that for $t\in[0,\delta_0]$, $F\geq1/2$ and
\begin{align}\label{4.11}
\begin{array}{l}
|F[v-Gw+\psi]|\geq|F|\cdot[|a_0-GH_0|-(|v|+|Gw|+|a_1|t)] \\
\qquad  \qquad  \qquad  \qquad \  \geq |F|\cdot[\eps_0-(Mt^2+GMt^2+Kt)]\geq\fr{\eps_0}{4}, \\
|F[u+Gw+\phi]|\geq|F|\cdot[|-a_0+GH_0|-(|u|+|Gw|+|a_1|t)] \\
\qquad  \qquad  \qquad  \qquad \  \geq |F|\cdot[\eps_0-(Mt^2+GMt^2+Kt)]\geq\fr{\eps_0}{4}, \\
|F[u+v+2a_1t]|\geq 2|F|(|a_1|-Mt)t\geq \fr{\eps_0}{2}t.
\end{array}
\end{align}
Moreover, there holds
\begin{align}\label{4.12}
|u|+|v|+|w|\leq Mt^2,\quad |u_r|+|v_r|+|w_r|\leq Mt^2,\quad |u_{rr}|+|v_{rr}|+|w_{rr}|\leq Mt^2.
\end{align}
\vspace{0.2cm}

For simplicity, denote $b_1$ as
\begin{align}\label{4.13}
b_1(u,v,w,t,r)=-\bigg(\fr{u+v}{2t}+a_1\bigg)\fr{I_1}{F}+\fr{t^2\sqrt{1-t^2}I_2}{F}+\fr{I_3 t}{F},
\end{align}
where $I_i$, $i=1,2,3$, denote
\begin{align*}
I_1&=\fr{(\kappa+1)(u+v+2a_1t)}{v-Gw+\psi}-\fr{(\kappa+2-t^2)(v+a_0+a_1t)
-(\kappa+1-t^2)G(w+H_0)}{v-Gw+\psi}t^2
 \nonumber \\[4pt]
I_2&=\fr{(-a_{0}'+a_{1}'t)(v+a_0+a_1t)}{v-Gw+\psi}
 \nonumber \\[4pt]
I_3&=\fr{(\kappa+2-2t^2)(u-a_0+a_1t)+(\kappa+1-t^2)G(w+H_0)}{v-Gw+\psi}(v+a_0+a_1t).
\end{align*}
According to \eqref{4.11} and \eqref{4.12}, we have
\begin{align}\label{4.14}
\begin{array}{l}
|I_1|\leq K(Mt^2+Kt)+K(Mt^2+K+Kt)t^2\leq K(1+Mt)t, \\
|I_2|\leq K(K+Kt)(Mt^2+K+Kt)\leq K(1+Mt),\\
|I_3|\leq K(Mt^2+K+Kt)(Mt^2+K+Kt)\leq K(1+Mt)^2.
\end{array}
\end{align}
Putting \eqref{4.14} into \eqref{4.13} yields
\begin{align}\label{4.15}
\begin{array}{l}
|b_1|\leq K(Mt+K)|I_1|+K|I_2|t^2+K|I_3|t\leq K(1+Mt)^2t.
\end{array}
\end{align}
Differentiation $b_1$ with respect to $r$, it renders
\begin{align*}
\fr{\pa b_1}{\pa r}=-\bigg(\fr{u_r+v_r}{2t}+a'_{1}\bigg)\fr{I_1}{F} -\bigg(\fr{u+v}{2t}+a_{1}\bigg)\fr{\pa_rI_1}{F} +\fr{t^2\sqrt{1-t^2}}{F}\pa_rI_2+\fr{t}{F}\pa_rI_3,
\end{align*}
and subsequently
\begin{align*}
\fr{\pa^2 b_1}{\pa r^2}=&-\bigg(\fr{u_{rr}+v_{rr}}{2t}+a''_{1}\bigg)\fr{I_1}{F} -\bigg(\fr{u+v}{2t}+a_{1}\bigg)\fr{\pa_{rr}I_1}{F} \\[4pt]
&\ -2\bigg(\fr{u_{r}+v_{r}}{2t}+a'_{1}\bigg)\fr{\pa_rI_1}{F}+\fr{t^2\sqrt{1-t^2}}{F}\pa_{rr}I_2+\fr{t}{F}\pa_{rr}I_3,
\end{align*}
Applying \eqref{4.11}-\eqref{4.12} and \eqref{4.14}, one obtains
\begin{align}\label{4.16}
\bigg|\fr{\pa b_1}{\pa r}\bigg|\leq K(1+Mt)^2t+K(1+Mt)|\pa_rI_1|+K|\pa_rI_2|t^2+K|\pa_rI_3|t,
\end{align}
and
\begin{align}\label{4.17}
\bigg|\fr{\pa^2 b_1}{\pa r^2}\bigg|\leq &K(1+Mt)^2t+K(1+Mt)|\pa_{rr}I_1| \nonumber \\
 &+K(1+Mt)|\pa_{r}I_1| +K|\pa_{rr}I_2|t^2+K|\pa_{rr}I_3|t.
\end{align}
By direct calculation and simplification, we derive
\begin{align}\label{4.18}
\pa_rI_1=&\fr{(\kappa+1)(u_r+v_r+2a'_1t)}{v-Gw+\psi} -\fr{(\kappa+1)(u+v+2a_1t)(v_r-Gw_r+\psi_r)}{(v-Gw+\psi)^2}\nonumber \\[4pt] &\ -\fr{(\kappa+2-t^2)(v_r+a'_0+a'_1t)-(\kappa+1-t^2)G(w_r+H'_0)}{v-Gw+\psi}t^2 \nonumber \\[4pt] &\
+\fr{(\kappa+2-t^2)(v+a_0+a_1t)-(\kappa+1-t^2)G(w+H_0)}{(v-Gw+\psi)^2}(v_r-Gw_r+\psi_r)t^2,
\end{align}
\begin{align}\label{4.19}
\pa_rI_2=&\fr{(-a''_0+a''_1t)(v+a_0+a_1t)}{v-Gw+\psi} +\fr{(-a'_0+a'_1t)(v_r+a'_0+a'_1t)}{v-Gw+\psi} \nonumber \\[4pt] &\ -\fr{(-a'_0+a'_1t)(v+a_0+a_1t)(v_r-Gw_r+\psi_r)}{(v-Gw+\psi)^2},
\end{align}
and
\begin{align}\label{4.20}
\pa_rI_3&=\fr{(\kappa+2-2t^2)(u_r-a'_0+a'_1t)+(\kappa+1-t^2)G(w_r+H'_0)}{v-Gw+\psi}(v+a_0+a_1t)  \nonumber \\[4pt] &\ -\fr{(\kappa+2-2t^2)(u-a_0+a_1t)+(\kappa+1-t^2)G(w+H_0)}{(v-Gw+\psi)^2}(v+a_0+a_1t)(v_r-Gw_r+\psi_r) \nonumber \\[4pt] &\ +\fr{(\kappa+2-2t^2)(u-a_0+a_1t)+(\kappa+1-t^2)G(w+H_0)}{v-Gw+\psi}(v_r+a'_0+a'_1t).
\end{align}
Moreover, one has
\begin{align}\label{4.21}
\pa_{rr}I_1=\fr{I_{11}}{v-Gw+\psi}+\fr{I_{12}}{(v-Gw+\psi)^2}+\fr{I_{13}}{(v-Gw+\psi)^3},
\end{align}
where $I_{1i}$, $i=1,2,3$, denote
\begin{align*}
I_{11}=&(\kappa+1)(u_{rr}+v_{rr}+2a''_1t)-(\kappa+2-t^2)(v_{rr}+a''_0+a''_1t)t^2 \\
 & +(\kappa+1-t^2)G(w_{rr}+H''_0)t^2, \\
I_{12}=&-2(\kappa+1)(u_r+v_r+2a'_1t)(v_r-Gw_r+\psi_r) \\
& -(\kappa+1)(u+v+2a_1t)(v_{rr}-Gw_{rr}+\psi_{rr}) \\
& +2[(\kappa+2-t^2)(v_r+a'_0+a'_1t)-(\kappa+1-t^2)G(w_r+H'_0)](v_r-Gw_r+\psi_r)t^2 \\
& +[(\kappa+2-t^2)(v+a_0+a_1t)-(\kappa+1-t^2)G(w+H_0)](v_{rr}-Gw_{rr}+\psi_{rr})t^2, \\
I_{13}=&2(\kappa+1)(u+v+2a_1t)(v_r-Gw_r+\psi_r)^2 \\ &-2[(\kappa+2-t^2)(v+a_0+a_1t)-(\kappa+1-t^2)G(w+H_0)](v_{r}-Gw_{r}+\psi_{r})^2t^2,
\end{align*}
\begin{align}\label{4.22}
\pa_{rr}I_2=\fr{I_{21}}{v-Gw+\psi}+\fr{I_{22}}{(v-Gw+\psi)^2}+\fr{I_{23}}{(v-Gw+\psi)^3},
\end{align}
where $I_{2i}$, $i=1,2,3$, denote
\begin{align*}
I_{21}=&(-a'''_0+a'''_1t)(v+a_0+a_1t)+2(-a''_0+a''_1t)(v_r+a'_0+a'_1t) \\
& +(-a'_0+a'_1t)(v_{rr}+a''_0+a''_1t), \\
I_{22}=&-2(-a''_0+a''_1t)(v+a_0+a_1t)(v_r-Gw_r+\psi_r) \\
& -2(-a'_0+a'_1t)(v_r+a'_0+a'_1t)(v_r-Gw_r+\psi_r) \\
&-(-a'_0+a'_1t)(v+a_0+a_1t)(v_{rr}-Gw_{rr}+\psi_{rr}), \\
I_{23}=&2(-a'_0+a'_1t)(v+a_0+a_1t)(v_r-Gw_r+\psi_r)^2,
\end{align*}
and
\begin{align}\label{4.23}
\pa_{rr}I_3=\fr{I_{31}}{v-Gw+\psi}+\fr{I_{32}}{(v-Gw+\psi)^2}+\fr{I_{33}}{(v-Gw+\psi)^3},
\end{align}
where $I_{3i}$, $i=1,2,3$, denote
\begin{align*}
I_{31}=&[(\kappa+2-2t^2)(u_{rr}-a''_0+a''_1t)+(\kappa+1-t^2)G(w_{rr}+H''_0)](v+a_0+a_1t) \\ &+2[(\kappa+2-2t^2)(u_r-a'_0+a'_1t)+(\kappa+1-t^2)G(w_r+H'_0)](v_r+a'_0+a'_1t) \\
&+[(\kappa+2-2t^2)(u-a_0+a_1t)+(\kappa+1-t^2)G(w+H_0)](v_{rr}+a''_0+a''_1t), \\
I_{32}=&-2[(\kappa+2-2t^2)(u_r-a'_0+a'_1t)+(\kappa+1-t^2)G(w_r+H'_0)]\\
&\quad \ \times(v+a_0+a_1t)(v_r-Gw_r+\psi_r) \\ &-[(\kappa+2-2t^2)(u-a_0+a_1t)+(\kappa+1-t^2)G(w+H_0)]\\
&\quad  \times[(v+a_0+a_1t)(v_{rr}-Gw_{rr}+\psi_{rr}) +2(v_r+a'_0+a'_1t)(v_r-Gw_r+\psi_r)], \\
I_{33}=&2[(\kappa+2-2t^2)(u-a_0+a_1t)+(\kappa+1-t^2)G(w+H_0)]\\
&\quad  \times(v+a_0+a_1t)(v_r-Gw_r+\psi_r)^2.
\end{align*}

Making use of \eqref{4.11} and \eqref{4.12}, we have estimates  from \eqref{4.18}-\eqref{4.20},
\begin{align}\label{4.24}
|\pa_rI_1|\leq &K(Mt^2+Kt)+K(Mt^2+Kt)(Mt^2+K)\nonumber \\
&+K[K(Mt^2+K+Kt)+K(Mt^2+K)]t^2 \nonumber \\
&\ +K[K(Mt^2+K+Kt)+K(Mt^2+K)](Mt^2+K)t^2 \nonumber \\
\leq & K(1+Mt)^2 t,
\end{align}
\begin{align}\label{4.25}
|\pa_rI_2|\leq &K(K+Kt)(Mt^2+K+Kt) +K(K+Kt)(Mt^2+K+Kt) \nonumber \\
&+K(K+Kt)(Mt^2+K+Kt)(Mt^2+K)\nonumber \\
\leq &K(1+Mt)^2,
\end{align}
and
\begin{align}\label{4.26}
|\pa_rI_3|\leq &K[K(Mt^2+K+Kt)+K(Mt^2+K)](Mt^2+K+Kt) \nonumber \\
&+K[K(Mt^2+K+Kt)+K(Mt^2+K)](Mt^2+K+Kt)(Mt^2+K) \nonumber \\
&+K[K(Mt^2+K+Kt)+K(Mt^2+K)](Mt^2+K+Kt)  \nonumber \\
\leq &K(1+Mt)^3.
\end{align}
Combining \eqref{4.16} and \eqref{4.24}-\eqref{4.26} suggests
\begin{align}\label{4.27}
\bigg|\fr{\pa b_1}{\pa r}\bigg|\leq &K(1+Mt)^2t+K(1+Mt)^3t+K(1+Mt)^2t^2+K(1+Mt)^3t \nonumber \\
\leq &K(1+Mt)^3t.
\end{align}
Furthermore, it is easy to  obtain the following estimates about $I_{ij}, (i,j=1,2,3)$
\begin{align}\label{4.28}
\left\{
\begin{array}{l}
|I_{11}|\leq K(1+Mt)t,\\
|I_{12}|\leq K(1+Mt)^2t,\\
|I_{13}|\leq K(1+Mt)^3t,\\
\end{array}
\right.\quad
\left\{
\begin{array}{l}
|I_{21}|\leq K(1+Mt),\\
|I_{22}|\leq K(1+Mt)^2,\\
|I_{23}|\leq K(1+Mt)^3,\\
\end{array}
\right.\quad
\left\{
\begin{array}{l}
|I_{31}|\leq K(1+Mt)^2,\\
|I_{32}|\leq K(1+Mt)^3,\\
|I_{33}|\leq K(1+Mt)^4.\\
\end{array}
\right.
\end{align}
Therefore, one obtains by \eqref{4.21}-\eqref{4.23} and \eqref{4.11},
$$
|\pa_{rr}I_1|\leq K(1+Mt)^3t,\quad  |\pa_{rr}I_2|\leq K(1+Mt)^3,\quad |\pa_{rr}I_1|\leq K(1+Mt)^4.
$$
We substitute the above into \eqref{4.17} and employ \eqref{4.24} to achieve
\begin{align}\label{4.29}
\bigg|\fr{\pa^2 b_1}{\pa r^2}\bigg|\leq &K(1+Mt)^2t+K(1+Mt)^4t \nonumber \\
 &+K(1+Mt)^3t +K(1+Mt)^3t^2+K(1+Mt)^4t \nonumber \\
 \leq&K(1+Mt)^4t.
\end{align}
Similar arguments apply for $b_2$ yield
\begin{align}\label{4.30}
|b_2|\leq K(1+Mt)^2t,\quad
\bigg|\fr{\pa b_2}{\pa r}\bigg|\leq K(1+Mt)^3t, \quad
\bigg|\fr{\pa^2 b_2}{\pa r^2}\bigg|\leq K(1+Mt)^4t.
\end{align}
Now we estimate $b_3, \pa_rb_3$ and $\pa_{rr}b_3$. Due to \eqref{4.11} and \eqref{4.12}, it is easily seen that there holds
\begin{align}\label{4.31}
|b_3|\leq\fr{Kt^2(Mt^2+K)}{\fr{1}{2}\eps_0t}\leq K(1+Mt)t.
\end{align}
We directly compute
\begin{align}\label{4.32}
\fr{\pa b_3}{\pa r}=-\fr{t^2\sqrt{1-t^2}}{F(t)}\bigg\{&\fr{(u_r-v_r-2a'_0)H'_0+(u-v-2a_0)H''_0}{u+v+2a_1t}
\nonumber \\[4pt]
&\quad \ -\fr{(u-v-2a_0)(u_r+v_r+2a'_1t)H'_0}{(u+v+2a_1t)^2}\bigg\},
\end{align}
and
\begin{align}\label{4.33}
\fr{\pa^2 b_3}{\pa r^2}=-\fr{t^2\sqrt{1-t^2}}{F(t)}\bigg\{&\fr{(u_{rr}-v_{rr}-2a''_0)H'_0+2(u_r-v_r-2a'_0)H''_0 +(u-v-2a_0)H'''_0}{u+v+2a_1t}
\nonumber \\[4pt]
&\ \ -\fr{2[(u_r-v_r-2a'_0)H'_0+(u-v-2a_0)H''_0](u_r+v_r+2a'_1t)}{(u+v+2a_1t)^2} \nonumber \\[4pt]
&\ \ \ -\fr{(u-v-2a_0)(u_{rr}+v_{rr}+2a''_1t)H'_0}{(u+v+2a_1t)^2}
\nonumber \\[4pt]
&\ \ \ \ +\fr{2(u-v-2a_0)(u_{r}+v_{r}+2a'_1t)^2H'_0}{(u+v+2a_1t)^3}
\bigg\}.
\end{align}
Using \eqref{4.11} and \eqref{4.12} again, we obtain
\begin{align}\label{4.34}
\bigg|\fr{\pa b_3}{\pa r}\bigg| &\leq Kt^2\bigg\{\fr{K(Mt^2+K)}{\fr{1}{2}\eps_0t}+\fr{K(Mt^2+K)(Mt^2+Kt)}{\fr{1}{4}\eps^{2}_0t^2}\bigg\} \nonumber \\[4pt]
&\leq K(1+Mt)^2t,
\end{align}
and
\begin{align}\label{4.35}
\bigg|\fr{\pa^2 b_3}{\pa r^2}\bigg|\leq &Kt^2\bigg\{\fr{K(Mt^2+K)}{\fr{1}{2}\eps_0t} +\fr{K(Mt^2+K)(Mt^2+Kt)}{\fr{1}{4}\eps^{2}_0t^2} \nonumber \\[4pt]
&\qquad \quad +\fr{K(Mt^2+K)(Mt^2+Kt)}{\fr{1}{4}\eps^{2}_0t^2} +\fr{K(Mt^2+K)(Mt^2+Kt)^2}{\fr{1}{8}\eps^{3}_0t^3}\bigg\} \nonumber \\[4pt]
\leq& K(1+Mt)^3t.
\end{align}
\vspace{0.2cm}

For  further use, we make the following  estimates about $\lambda_i(u,v,w,t,r) (i=1,2,3)$. It follows from \eqref{4.11} and \eqref{4.12} that
\begin{align}\label{4.36}
|\lambda_1|&=\bigg|-\fr{t^2\sqrt{1-t^2}}{F}\cdot\fr{v+a_0+a_1t}{v-Gw+\psi}\bigg| \nonumber \\[4pt]
&\leq K(1+Mt)t^2.
\end{align}
By performing a direct calculation, we obtain
\begin{align*}
\fr{\pa\lambda_1}{\pa r}=-\fr{t^2\sqrt{1-t^2}}{F}\bigg\{\fr{v_r+a'_0+a'_1t}{v-Gw+\psi} -\fr{(v+a_0+a_1t)(v_r-Gw_r+\psi_r)}{(v-Gw+\psi)^2}\bigg\}
\end{align*}
and
\begin{align*}
\fr{\pa^2\lambda_1}{\pa r^2}=-\fr{t^2\sqrt{1-t^2}}{F}&\bigg\{\fr{v_{rr}+a''_0+a''_1t}{v-Gw+\psi} -\fr{2(v_r+a'_0+a'_1t)(v_r-Gw_r+\psi_r)}{(v-Gw+\psi)^2} \nonumber \\[4pt]
&\quad -\fr{(v+a_0+a_1t)(v_{rr}-Gw_{rr}+\psi_{rr})}{(v-Gw+\psi)^2} \nonumber \\[4pt]
&\quad \ \ +\fr{2(v+a_0+a_1t)(v_{r}-Gw_{r}+\psi_{r})^2}{(v-Gw+\psi)^3}
\bigg\}.
\end{align*}
Then we have estimates
\begin{align}\label{4.37}
\bigg|\fr{\pa\lambda_1}{\pa r}\bigg|&\leq Kt^2\bigg\{K(Mt^2+K+Kt)+K(Mt^2+K+Kt)(Mt^2+K)\bigg\} \nonumber \\[4pt]
&\leq K(1+Mt)^2t^2,
\end{align}
and
\begin{align}\label{4.38}
\bigg|\fr{\pa^2\lambda_1}{\pa r^2}\bigg|&\leq Kt^2\bigg\{K(Mt^2+K+Kt)+K(Mt^2+K+Kt)(Mt^2+K) \nonumber \\[4pt] &\qquad \qquad +K(Mt^2+K+Kt)(Mt^2+K)+K(Mt^2+K+Kt)(Mt^2+K)^2 \bigg\} \nonumber \\[4pt]
&\leq K(1+Mt)^3t^2.
\end{align}
Repetition of the same arguments for $\lambda_2$ leads to
\begin{align}\label{4.39}
|\lambda_2|\leq K(1+Mt)t^2,\quad \bigg|\fr{\pa\lambda_2}{\pa r}\bigg|\leq K(1+Mt)^2t^2, \quad \bigg|\fr{\pa^2\lambda_2}{\pa r^2}\bigg|\leq K(1+Mt)^3t^2.
\end{align}
For $\lambda_3$, one has
\begin{align}\label{4.40}
|\lambda_3|=\bigg|\fr{t^2\sqrt{1-t^2}}{F}\cdot\fr{u-v-2a_0}{u+v+2a_1t}\bigg|
\leq Kt^2\fr{Mt^2+K}{\fr{1}{2}\eps_0t}\leq K(1+Mt)t,
\end{align}
\begin{align}\label{4.41}
\bigg|\fr{\pa \lambda_3}{\pa r}\bigg|&=\fr{t^2\sqrt{1-t^2}}{F}\cdot\bigg|\fr{u_r-v_r-2a'_0}{u+v+2a_1t} -\fr{(u-v-2a_0)(u_r+v_r+2a'_1t)}{(u+v+2a_1t)^2}\bigg| \nonumber \\[4pt]
&\leq Kt^2\bigg\{\fr{Mt^2+K}{\fr{1}{2}\eps_0t} +\fr{K(Mt^2+K)(Mt^2+Kt)}{\fr{1}{4}\eps^{2}_0t^2}\bigg\} \nonumber \\[4pt]
&\leq K(1+Mt)^2t,
\end{align}
and
\begin{align}\label{4.42}
\bigg|\fr{\pa^2 \lambda_3}{\pa r^2}\bigg|&= \fr{t^2\sqrt{1-t^2}}{F}\cdot
\bigg|\fr{u_{rr}-v_{rr}-2a''_0}{u+v+2a_1t} -\fr{2(u_r-v_r-2a'_0)(u_r+v_r+2a'_1t)}{(u+v+2a_1t)^2} \nonumber \\[4pt] &\qquad  \qquad \qquad \quad -\fr{(u-v-2a_0)(u_{rr}+v_{rr}+2a''_1t)}{(u+v+2a_1t)^2} \nonumber \\[4pt] &\qquad  \qquad \qquad \qquad  +\fr{2(u-v-2a_0)(u_r+v_r+2a'_1t)^2}{(u+v+2a_1t)^3}\bigg| \nonumber \\[4pt]
&\leq Kt^2\bigg\{\fr{Mt^2+K}{\fr{1}{2}\eps_0t} +\fr{K(Mt^2+K)(Mt^2+Kt)}{\fr{1}{4}\eps^{2}_0t^2} \nonumber \\[4pt]
&\qquad \qquad + \fr{K(Mt^2+K)(Mt^2+Kt)}{\fr{1}{4}\eps^{2}_0t^2} +\fr{K(Mt^2+K)(Mt^2+Kt)^2}{\fr{1}{8}\eps^{3}_0t^3}\bigg\} \nonumber \\[4pt]
&\leq K(1+Mt)^3t.
\end{align}
Summing up \eqref{4.15}, \eqref{4.27}, \eqref{4.29}-\eqref{4.31}, \eqref{4.34}-\eqref{4.42}, we have the following a priori estimates
\begin{align}\label{4.43}
\begin{array}{l}
|b_i|\leq K(1+Mt)^2t,\quad |\fr{\pa b_i}{\pa r}|\leq K(1+Mt)^3t, \quad |\fr{\pa^2 b_i}{\pa r^2}|\leq K(1+Mt)^4t,
\\
|\lambda_i|\leq K(1+Mt)t,\quad |\fr{\pa\lambda_i}{\pa r}|\leq K(1+Mt)^2t, \quad |\fr{\pa^2\lambda_i}{\pa r^2}|\leq K(1+Mt)^3t,
\end{array}
\quad i=1,2,3.
\end{align}
\vspace{0.2cm}

\noindent{\bf Step 3 (Properties of the mapping). } We now study the properties of the mapping $\mathcal{T}$.
\begin{lem}\label{lem1}
Let the assumptions in Theorem \ref{thm3} hold. Then there exists positive constants $\delta\in(0,\delta_0), M$ and $0<\nu<1$ depending only on $\eps_0, \kappa$, the bounds of $F, G$ and the $C^3$ norms of $a_0, a_1, H_0$ such that

\noindent (1) $\mathcal{T}$ maps $\mathcal{S}^{M}_\delta$ into $\mathcal{S}^{M}_\delta$;

\noindent (2) For any vector functions $\mathbf{F}, \widehat{\mathbf{F}}$ in $\mathcal{S}^{M}_\delta$, there holds
\begin{align}\label{4.44}
d\bigg(\mathcal{T}(\mathbf{F}),\mathcal{T}(\widehat{\mathbf{F}})\bigg)\leq\nu d(\mathbf{F},\widehat{\mathbf{F}}).
\end{align}
\end{lem}
\begin{proof}
Assume $\mathbf{F}=(u,v,w)^T$ and $\widehat{\mathbf{F}}=(\widehat{u},\widehat{v}, \widehat{w})^T$ are two elements in $\mathcal{S}^{M}_\delta$, where the constants $\delta$ and $M$ will be determined later. Denote $\mathbf{G}=\mathcal{T}(\mathbf{F})=(U,V,W)^T$ and $\widehat{\mathbf{G}}=\mathcal{T}(\widehat{\mathbf{F}})=(\widehat{U},\widehat{V},\widehat{W})^T$. It is easy to check by \eqref{4.9} that $U(0,\eta)=V(0,\eta)=W(0,\eta)=0$.
\vspace{0.2cm}

According to \eqref{4.9} and \eqref{4.43}, it provides that for $t\leq\delta$
\begin{align}\label{4.45}
|U(\xi,\eta)|&=\bigg|\int_{0}^\xi\bigg(\fr{u+v}{2t}+\widetilde{b}_1\bigg)\ {\rm d}t\bigg| \nonumber \\[4pt]
&\leq\int_{0}^\xi\bigg(\fr{|u|+|v|}{2t}+|\widetilde{b}_1|\bigg)\ {\rm d}t \nonumber \\[4pt]
&\leq\int_{0}^\xi\bigg(\fr{M}{2}t+K(1+M\delta)^2t\bigg)\ {\rm d}t\leq \bigg(\fr{M}{4}+K(1+M\delta)^2\bigg)\xi^2.
\end{align}
Similarly, one has
\begin{align}\label{4.46}
|V(\xi,\eta)|\leq \bigg(\fr{M}{4}+K(1+M\delta)^2\bigg)\xi^2,\quad  |W(\xi,\eta)|\leq K(1+M\delta)^2\xi^2.
\end{align}
Combining \eqref{4.45} and \eqref{4.46} yields
\begin{align}\label{4.47}
\bigg|\fr{U(\xi,\eta)}{\xi^2}\bigg|+\bigg|\fr{V(\xi,\eta)}{\xi^2}\bigg| +\bigg|\fr{W(\xi,\eta)}{\xi^2}\bigg|\leq\fr{M}{2}+K(1+M\delta)^2.
\end{align}
To establish the bound of $U_\eta/\xi^2$, we differentiate $U(\xi,\eta)$ with respect to $\eta$
\begin{align}\label{4.48}
\fr{\pa U}{\pa \eta}(\xi,\eta)=\int_{0}^\xi\bigg(\fr{u_r+v_r}{2t}+\fr{\pa b_1}{\pa r}\bigg)\cdot\fr{\pa r_1}{\pa \eta}\ {\rm d}t,
\end{align}
where
$$
\fr{\pa r_1}{\pa \eta}(t;\xi,\eta)=\exp\bigg(\int_{\xi}^t\fr{\pa \lambda_1}{\pa r}(\tau,r_1(\tau;\xi,\eta))\ {\rm d}\tau\bigg).
$$
We use  \eqref{4.43} to estimate
\begin{align}\label{4.49}
\bigg|\fr{\pa r_1}{\pa \eta}\bigg|\leq\exp\bigg(\int_{0}^t\bigg|\fr{\pa \lambda_1}{\pa r}\bigg|\ {\rm d}\tau\bigg) \leq\exp\bigg(\int_{0}^\delta K(1+M\delta)^2\tau\ {\rm d}\tau\bigg)\leq e^{K(1+M\delta)^2\delta^2}.
\end{align}
Employing \eqref{4.43} again, we obtain
\begin{align}\label{4.50}
\bigg|\fr{\pa U}{\pa \eta}(\xi,\eta)\bigg|&\leq\int_{0}^\xi\bigg(\fr{|u_r|+|v_r|}{2t}+\bigg|\fr{\pa b_1}{\pa r}\bigg|\bigg)\cdot\bigg|\fr{\pa r_1}{\pa \eta}\bigg|\ {\rm d}t \nonumber \\[4pt]
&\leq \int_{0}^\xi\bigg(\fr{M}{2}t+K(1+M\delta)^3t\bigg)e^{K(1+M\delta)^2\delta^2}\ {\rm d}t
\nonumber \\[4pt]
&\leq \xi^2\bigg(\fr{M}{4}+K(1+M\delta)^3\bigg)e^{K(1+M\delta)^2\delta^2}.
\end{align}
In a similar way, we also obtain
\begin{align*}
\bigg|\fr{\pa V}{\pa \eta}(\xi,\eta)\bigg|&\leq \xi^2\bigg(\fr{M}{4}+K(1+M\delta)^3\bigg)e^{K(1+M\delta)^2\delta^2}, \\[4pt]
\bigg|\fr{\pa W}{\pa \eta}(\xi,\eta)\bigg|&\leq \xi^2 K(1+M\delta)^3e^{K(1+M\delta)^2\delta^2},
\end{align*}
which, together with \eqref{4.50}, provides
\begin{align}\label{4.51}
\bigg|\fr{ U_\eta(\xi,\eta)}{\xi^2}\bigg| +\bigg|\fr{ V_\eta(\xi,\eta)}{\xi^2}\bigg| +\bigg|\fr{ W_\eta(\xi,\eta)}{\xi^2}\bigg| \leq \bigg(\fr{M}{2}+K(1+M\delta)^3\bigg)e^{K(1+M\delta)^2\delta^2}.
\end{align}
Now, differentiating \eqref{4.48} with respect to $\eta$ gives
\begin{align}\label{4.52}
\fr{\pa^2 U}{\pa \eta^2}(\xi,\eta)=\int_{0}^\xi\bigg\{\bigg(\fr{u_{rr}+v_{rr}}{2t}+\fr{\pa^2 b_1}{\pa r^2}\bigg)\cdot\bigg(\fr{\pa r_1}{\pa \eta}\bigg)^2 +\bigg(\fr{u_{r}+v_{r}}{2t}+\fr{\pa b_1}{\pa r}\bigg)\cdot\fr{\pa^2 r_1}{\pa \eta^2}\bigg\} \ {\rm d}t,
\end{align}
where
\begin{align*}
\fr{\pa^2 r_1}{\pa \eta^2}(t;\xi,\eta)=&\exp\bigg(\int_{\xi}^t\fr{\pa \lambda_1}{\pa r}(\tau,r_1(\tau;\xi,\eta))\ {\rm d}\tau\bigg) \\[4pt]
&\quad \times
\int_{\xi}^t\fr{\pa^2 \lambda_1}{\pa r^2}\cdot\fr{\pa r_1}{\pa \eta}(\tau,r_1(\tau;\xi,\eta))\ {\rm d}\tau.
\end{align*}
We use \eqref{4.43} and recall \eqref{4.49} to see that
\begin{align}\label{4.53}
\bigg|\fr{\pa^2 r_1}{\pa \eta^2}\bigg|&\leq \exp\bigg(\int_{0}^t\bigg|\fr{\pa \lambda_1}{\pa r}\bigg|\ {\rm d}\tau\bigg)\times\int_{0}^t\bigg|\fr{\pa^2 \lambda_1}{\pa r^2}\bigg|\cdot\bigg|\fr{\pa r_1}{\pa \eta}\bigg|\ {\rm d}\tau \nonumber \\[4pt]
&\leq \exp\bigg(\int_{0}^\delta K(1+M\delta)^2\tau\ {\rm d}\tau\bigg)\times\int_{0}^\delta K(1+M\delta)^3\tau\cdot e^{K(1+M\delta)^2\delta^2}\ {\rm d}\tau  \nonumber \\[4pt]
&\leq K\delta^2(1+M\delta)^3e^{K(1+M\delta)^2\delta^2}.
\end{align}
Inserting \eqref{4.53} into \eqref{4.52} and applying \eqref{4.43} again results in
\begin{align}\label{4.54}
\bigg|\fr{\pa^2 U}{\pa \eta^2}(\xi,\eta)\bigg|&\leq\int_{0}^\xi\bigg\{\bigg(\fr{|u_{rr}|+|v_{rr}|}{2t}+\bigg|\fr{\pa^2 b_1}{\pa r^2}\bigg|\bigg)\cdot\bigg|\fr{\pa r_1}{\pa \eta}\bigg|^2  \nonumber \\[4pt]
&\qquad \qquad   \ +\bigg(\fr{|u_{r}|+|v_{r}|}{2t}+\bigg|\fr{\pa b_1}{\pa r}\bigg|\bigg)\cdot\bigg|\fr{\pa^2 r_1}{\pa \eta^2}\bigg|\bigg\} \ {\rm d}t \nonumber \\[4pt]
&\leq \int_{0}^\xi\bigg\{\bigg(\fr{M}{2}t+K(1+M\delta)^4t\bigg)e^{K(1+M\delta)^2\delta^2} \nonumber \\[4pt]
&\qquad \qquad   \ +\bigg(\fr{M}{2}t+K(1+M\delta)^3t\bigg)K\delta^2(1+M\delta)^3e^{K(1+M\delta)^2\delta^2}\bigg\}\ {\rm d}t \nonumber \\[4pt]
&\leq \xi^2\bigg(\fr{M}{4}+K(1+M\delta)^4\bigg)[1+K\delta^2(1+M\delta)^3]e^{K(1+M\delta)^2\delta^2}.
\end{align}
With similar arguments, one proceeds to obtain
\begin{align*}
\bigg|\fr{\pa^2 V}{\pa \eta^2}(\xi,\eta)\bigg|&\leq \xi^2\bigg(\fr{M}{4}+K(1+M\delta)^4\bigg)[1+K\delta^2(1+M\delta)^3]e^{K(1+M\delta)^2\delta^2}, \\[4pt]
\bigg|\fr{\pa^2 W}{\pa \eta^2}(\xi,\eta)\bigg|&\leq \xi^2K(1+M\delta)^4[1+K\delta^2(1+M\delta)^3]e^{K(1+M\delta)^2\delta^2},
\end{align*}
and
\begin{align}\label{4.55}
&\bigg|\fr{ U_{\eta\eta}(\xi,\eta)}{\xi^2}\bigg| +\bigg|\fr{ V_{\eta\eta}(\xi,\eta)}{\xi^2}\bigg| +\bigg|\fr{ W_{\eta\eta}(\xi,\eta)}{\xi^2}\bigg| \nonumber \\[4pt]
\leq &\bigg(\fr{M}{2}+K(1+M\delta)^4\bigg)[1+K\delta^2(1+M\delta)^3]e^{K(1+M\delta)^2\delta^2}.
\end{align}
\vspace{0.2cm}

By choosing $M\geq64K\geq64$ and letting $\delta\leq\min\{1/M,\delta_0\}$, we observe
\begin{align*}
&\bigg(\fr{M}{2}+K(1+M\delta)^4\bigg)[1+K\delta^2(1+M\delta)^3]e^{K(1+M\delta)^2\delta^2}
\\[4pt]
\leq&\bigg(\fr{M}{2}+16K\bigg)(1+8K\delta^2)e^{4K\delta^2} \\[4pt]
\leq&\bigg(\fr{M}{2}+\fr{M}{4}\bigg)\bigg(1+\fr{\delta}{8}\bigg) e^{\fr{\delta}{16}}\leq\fr{27}{32}e^{\fr{1}{16}}M<M,
\end{align*}
which,  along with \eqref{4.47}, \eqref{4.51} and \eqref{4.55}, indicates that (P2)-(P4) are preserved by the mapping $\mathcal{T}$. To determine $\mathcal{T}(\mathbf{F})\in\mathcal{S}^{M}_\delta$, it suffices to show that $U_\xi(0,\eta)=V_\xi(0,\eta)=W_\xi(0,\eta)=0$. We differentiate the first equation of \eqref{4.9} with respect to $\xi$ to arrive at
\begin{align}\label{4.56}
\fr{\pa U}{\pa \xi}(\xi, \eta)=\fr{u+v}{2\xi}+b_1+\int_{0}^\xi\bigg(\fr{u_r+v_r}{2t}+\fr{\pa b_1}{\pa r}\bigg)\cdot\fr{\pa r_1}{\pa \xi}\ {\rm d}t,
\end{align}
where
\begin{align}\label{4.57}
\fr{\pa r_1}{\pa \xi}(t;\xi,\eta)=-\lambda_1\fr{\pa r_1}{\pa \eta}(t;\xi,\eta).
\end{align}
From \eqref{4.56} and \eqref{4.43}, we get $U_\xi(0,\eta)=0$. Similarly, we have $V_\xi(0,\eta)=W_\xi(0,\eta)=0$, which implies that $\mathcal{T}$ do map $\mathcal{S}^{M}_\delta$ onto itself.
\vspace{0.2cm}

Next we establish \eqref{4.44} for some positive constant $\nu<1$. It follows by the first equation of \eqref{4.8} that
\begin{align*}
\fr{\rm d}{{\rm d}_1t}U&=\fr{u+v}{2t}+b_1(u,v,w,t,r), \\[5pt]
\fr{\rm d}{{\rm d}_{\hat{1}}t}\widehat{U}
&=\fr{\widehat{u}+\widehat{v}}{2t}+b_1(\widehat{u},\widehat{v},\widehat{w},t,r),
\end{align*}
from which and \eqref{4.7} we obtain
\begin{align}\label{4.58}
\fr{\rm d}{{\rm d}_1t}(U-\widehat{U})&=\fr{\rm d}{{\rm d}_1t}U-\bigg(\fr{\rm d}{{\rm d}_{\hat{1}}t}\widehat{U} +[\lambda_1(\widehat{u},\widehat{v},\widehat{w},t,r)-\lambda_1(u,v,w,t,r)]\widehat{U}_r\bigg) \nonumber \\[4pt]
&=\fr{(u-\widehat{u})+(v-\widehat{v})}{2t}+[b_1(u,v,w,t,r)-b_1(\widehat{u},\widehat{v},\widehat{w},t,r)]  \nonumber \\[4pt] &\qquad +[\lambda_1(u,v,w,t,r)-\lambda_1(\widehat{u},\widehat{v},\widehat{w},t,r)]\widehat{U}_r \nonumber \\[4pt]
&:=I_4+I_5+I_6.
\end{align}
Obviously, there hold
\begin{align}\label{4.59}
|I_4|\leq \fr{|u-\widehat{u}|+|v-\widehat{v}|}{2t}\leq \fr{t}{2}d(\mathbf{F},\widehat{\mathbf{F}}),
\end{align}
and
\begin{align}\label{4.60}
|I_6|&\leq |\lambda_1(u,v,w,t,r)-\lambda_1(\widehat{u},\widehat{v},\widehat{w},t,r)|\cdot|\widehat{U}_r|\nonumber \\[4pt]
&\leq\fr{t^2\sqrt{1-t^2}}{F}\cdot\fr{|\psi-(a_0+a_1t)-G\widehat{w}|\cdot|v-\widehat{v}| +G|\widehat{v}+a_0+a_1t|\cdot|w-\widehat{w}|}{|v-Gw+\psi|\cdot|\widehat{v}-G\widehat{w}+\psi|}|\widehat{U}_r|
\nonumber \\[4pt]
&\leq KM(1+M\delta)t^6d(\mathbf{F},\widehat{\mathbf{F}}).
\end{align}
Recalling \eqref{4.13}, we have
\begin{align}\label{4.61}
|I_5|&\leq \bigg|\fr{u+v}{2t}+a_1\bigg|\cdot\fr{|I_1-\widehat{I}_1|}{F} +\fr{|\widehat{I}_1|}{F}\cdot\fr{|u-\widehat{u}|+|v-\widehat{v}|}{2t}  \nonumber \\[4pt] &\qquad +\fr{t^2\sqrt{1-t^2}}{F}|I_2-\widehat{I}_2| +\fr{t}{F}|I_3-\widehat{I}_3|,
\end{align}
where  $I_i$, $i=1,2,3$, are defined in \eqref{4.13} and $\widehat{I}_i$, $i=1,2,3$,  are the terms obtained by replacing
$(u,v,w)$ with $(\widehat{u},\widehat{v},\widehat{w})$ in $I_i$. We estimate $|I_i-\widehat{I}_i|$, $i=1,2,3$, in the following. Due to the expressions of $I_i (i=1,2,3)$, one derives
\begin{align}\label{4.62}
&|I_1-\widehat{I}_1|\leq (\kappa+1)\bigg|\fr{u+v+2a_1t}{v-Gw+\psi} -\fr{\widehat{u}+\widehat{v}+2a_1t}{\widehat{v}-G\widehat{w}+\psi}\bigg| \nonumber \\[4pt] &\qquad \qquad \quad  +(\kappa+2-t^2)t^2\bigg|\fr{v+a_0+a_1t}{v-Gw+\psi}-\fr{\widehat{v}+a_0+a_1t}{\widehat{v}-G\widehat{w}+\psi}\bigg| \nonumber \\[4pt] &\qquad \qquad \quad  \ +(\kappa+1-t^2)t^2G\bigg|\fr{w+H_0}{v-Gw+\psi}-\fr{\widehat{w}+H_0}{\widehat{v}-G\widehat{w}+\psi}\bigg| \nonumber \\
\leq& K(|\widehat{v}-G\widehat{w}+\psi|\cdot|u-\widehat{u}|+|\psi-\widehat{u}-G\widehat{w}-2a_1t|\cdot|v-\widehat{v}| +G|\widehat{u}+\widehat{v}+2a_1t|\cdot|w-\widehat{w}|) \nonumber \\
&\ +Kt^2(|\psi-a_0-a_1t-G\widehat{w}|\cdot|v-\widehat{v}|+G|\widehat{v}+a_0+a_1t|\cdot|v-\widehat{v}|) \nonumber \\
&\ \ +Kt^2(|\widehat{w}+H_0|\cdot|v-\widehat{v}|+|\widehat{v}+\psi+GH_0|\cdot|w-\widehat{w}|) \nonumber \\
\leq& K(1+M\delta)t^2d(\mathbf{F},\widehat{\mathbf{F}}),
\end{align}
\begin{align}\label{4.63}
|I_2-\widehat{I}_2|&\leq |-a'_0+a'_1t|\cdot\bigg|\fr{v+a_0+a_1t}{v-Gw+\psi} -\fr{\widehat{v}+a_0+a_1t}{\widehat{v}-G\widehat{w}+\psi}\bigg| \nonumber \\
&\leq K(|\psi-G\widehat{w}-a_0-a_1t|\cdot|v-\widehat{v}| +G|\widehat{v}+a_0+a_1t|\cdot|w-\widehat{w}|)
\nonumber \\
& \leq K(1+M\delta)t^2d(\mathbf{F},\widehat{\mathbf{F}}),
\end{align}
and
\begin{align}\label{4.64}
|I_3-\widehat{I}_3|\leq &(\kappa+2-2t^2)\bigg|\fr{(u-a_0+a_1t)(v+a_0+a_1t)}{v-Gw+\psi} -\fr{(\widehat{u}-a_0+a_1t)(\widehat{v}+a_0+a_1t)}{\widehat{v}-G\widehat{w}+\psi}\bigg| \nonumber \\[4pt]
&+(\kappa+a-t^2)G\bigg|\fr{(w+H_0)(v+a_0+a_1t)}{v-Gw+\psi} -\fr{(\widehat{w}+H_0)(\widehat{v}+a_0+a_1t)}{\widehat{v}-G\widehat{w}+\psi}\bigg| \nonumber \\
\leq& K\bigg\{|(v+a_0+a_1t)(\widehat{v}-G\widehat{w}+\psi)|\cdot|u-\widehat{u}| \nonumber \\
&\qquad +|(\widehat{u}-a_0+a_1t)(\psi-G\widehat{w}-a_0-a_1t)|\cdot|v-\widehat{v}| \nonumber \\
&\qquad \ +G|(\widehat{u}-a_0+a_1t)(\widehat{v}+a_0+a_1t)|\cdot|w-\widehat{w}|\bigg\} \nonumber \\
&+ K\bigg\{|(\widehat{w}+H_0)(\psi-Gw-a_0-a_1t|)|\cdot|v-\widehat{v}| \nonumber \\ &\qquad \quad +|(v+a_0+a_1t)(\widehat{v}+H_0G+\psi)|\cdot|w-\widehat{w}|\bigg\} \nonumber \\
\leq & K(1+M\delta)^2t^2d(\mathbf{F},\widehat{\mathbf{F}}).
\end{align}
Inserting \eqref{4.62}-\eqref{4.64} into \eqref{4.61} and recalling \eqref{4.14} acquires
\begin{align}\label{4.65}
|I_5|&\leq (K+M\delta)\cdot K(1+M\delta)t^2d(\mathbf{F},\widehat{\mathbf{F}}) +K(1+M\delta)t\cdot td(\mathbf{F},\widehat{\mathbf{F}}) \nonumber \\
&\quad +Kt^2\cdot K(1+M\delta)t^2d(\mathbf{F},\widehat{\mathbf{F}}) +Kt\cdot K(1+M\delta)^2t^2d(\mathbf{F},\widehat{\mathbf{F}}) \nonumber \\
&\leq K(1+M\delta)^2t^2d(\mathbf{F},\widehat{\mathbf{F}}).
\end{align}
We combine \eqref{4.58}-\eqref{4.60} and \eqref{4.65} to achieve
\begin{align*}
|U-\widehat{U}|&\leq\int_{0}^\xi(|I_4|+|I_5|+|I_6|)\ {\rm d}t \nonumber \\[4pt]
&\leq \int_{0}^\xi t \bigg\{\fr{1}{2}+K(1+M\delta)^2t+KM(1+M\delta)t^5\bigg\} d(\mathbf{F},\widehat{\mathbf{F}})\ {\rm d}t
\nonumber \\[5pt]
&\leq \int_{0}^\xi t \bigg\{\fr{1}{2}+K(1+M\delta)^2\delta+KM(1+M\delta)\delta^5\bigg\} d(\mathbf{F},\widehat{\mathbf{F}})\ {\rm d}t
\nonumber \\[5pt]
&\leq \fr{1}{2}\xi^2 \bigg\{\fr{1}{2}+K(1+M\delta)^2\delta+KM(1+M\delta)\delta^5\bigg\} d(\mathbf{F},\widehat{\mathbf{F}}),
\end{align*}
from which we obtain
\begin{align}\label{4.66}
\fr{|U-\widehat{U}|}{\xi^2}\leq \fr{1}{2}\bigg\{\fr{1}{2}+K(1+M\delta)^2\delta+KM(1+M\delta)\delta^5\bigg\} d(\mathbf{F},\widehat{\mathbf{F}}).
\end{align}
Due to the symmetry, we use the same argument as above for $V$ to have
\begin{align}\label{4.67}
\fr{|V-\widehat{V}|}{\xi^2}\leq \fr{1}{2}\bigg\{\fr{1}{2}+K(1+M\delta)^2\delta+KM(1+M\delta)\delta^5\bigg\} d(\mathbf{F},\widehat{\mathbf{F}}).
\end{align}
For the term $(W-\widehat{W})$, it follows that
\begin{align}\label{4.68}
\fr{\rm d}{{\rm d}_3t}(W-\widehat{W})&=\fr{\rm d}{{\rm d}_3t}W-\bigg(\fr{\rm d}{{\rm d}_{\hat{3}}t}\widehat{W} +[\lambda_3(\widehat{u},\widehat{v},\widehat{w},t,r)-\lambda_3(u,v,w,t,r)]\widehat{W}_r\bigg) \nonumber \\[4pt]
&=[b_3(u,v,w,t,r)-b_3(\widehat{u},\widehat{v},\widehat{w},t,r)]  \nonumber \\[4pt] &\qquad +[\lambda_3(u,v,w,t,r)-\lambda_3(\widehat{u},\widehat{v},\widehat{w},t,r)]\widehat{W}_r \nonumber \\[4pt]
&:=I_7+I_8.
\end{align}
Recalling the definitions  of $b_3$ and $\lambda_3$, we compute
\begin{align*}
I_7=-\fr{2t^2\sqrt{1-t^2}H'_0}{F(u+v+2a_1t)(\widehat{u}+\widehat{v}+2a_1t)}\{(\widehat{v}+a_0 +a_1t)(u-\widehat{u})-(\widehat{u}-a_0 +a_1t)(v-\widehat{v})\},
\end{align*}
and
\begin{align*}
I_8=\fr{2t^2\sqrt{1-t^2}\widehat{W}_r}{F(u+v+2a_1t)(\widehat{u}+\widehat{v}+2a_1t)}\{(\widehat{v}+a_0 +a_1t)(u-\widehat{u})-(\widehat{u}-a_0 +a_1t)(v-\widehat{v})\},
\end{align*}
from which, together with  \eqref{4.11} and  \eqref{4.43}, we obtain
$$
|I_7|\leq K(1+M\delta)t^2d(\mathbf{F},\widehat{\mathbf{F}}),\quad |I_8|\leq K(1+M\delta)t^4d(\mathbf{F},\widehat{\mathbf{F}}).
$$
Putting the above into \eqref{4.68} yields
\begin{align*}
|W-\widehat{W}|\leq\int_{0}^\xi(|I_7|+|I_8|)\ {\rm d}t\leq K(1+M\delta)\xi^3d(\mathbf{F},\widehat{\mathbf{F}}).
\end{align*}
That is, there holds
\begin{align}\label{4.69}
\fr{|W-\widehat{W}|}{\xi^2}\leq K(1+M\delta)\delta d(\mathbf{F},\widehat{\mathbf{F}}).
\end{align}
We add \eqref{4.66}, \eqref{4.67} and \eqref{4.69} to find
\begin{align*}
&\fr{|U-\widehat{U}|}{\xi^2}+\fr{|V-\widehat{V}|}{\xi^2}+\fr{|W-\widehat{W}|}{\xi^2} \nonumber \\[4pt]
\leq &\bigg\{\fr{1}{2} +K(1+M\delta)\delta +K(1+M\delta)^2\delta+KM(1+M\delta)\delta^5\bigg\}d(\mathbf{F},\widehat{\mathbf{F}}) \nonumber \\[4pt]
=:&\nu d(\mathbf{F},\widehat{\mathbf{F}})
\end{align*}
for $\nu<1$ if $\delta$ is chosen as before, which completes the proof of \eqref{4.44}. Thus $\mathcal{T}$ is a contraction under the metric $d$.
\end{proof}

\noindent{\bf Step 4 (Properties of the limit function).}  Since $(\mathcal{S}^{M}_\delta,d)$ is not a closed subset in the complete space $(\mathcal{W}^{M}_\delta,d)$, we need to confirm that the limit vector function of the iteration sequence $\{\mathbf{F}^{(n)}\}$, defined by $\mathbf{F}^{(n)}=\mathcal{T}\mathbf{F}^{(n-1)}$, is in $\mathcal{S}^{M}_\delta$. This follows directly from Arzela-Ascoli Theorem and the following lemma.
\vspace{0.2cm}

\begin{lem}\label{lem2}
With the assumptions in Theorem \ref{thm3},   the iteration sequence $\{\mathbf{F}^{(n)}\}$ has the property  that $\{\pa_t \mathbf{F}^{(n)}(t,r)\}$ and $\{\pa_r\mathbf{F}^{(n)}(t,r)\}$ are uniformly Lipschitz continuous on $D_\delta$.
\end{lem}
\begin{proof}
Assume $(u,v,w)^T\in\mathcal{S}^{M}_\delta$. Then, due to Lemma \ref{lem1}, we know that $(U,V,W)^T=\mathcal{T}(u,v,w)^T$ also in $\mathcal{S}^{M}_\delta$. The proof of this lemma will be given in three steps.
\vspace{0.2cm}

\noindent {\bf Step 1.  To show $|U_t|+|V_t|+|W_t|\leq 2Mt$. } To prove this, we recall \eqref{4.56} and use \eqref{4.43}, \eqref{4.49} and \eqref{4.57} to observe
\begin{align}\label{4.70}
&\bigg|\fr{\pa U}{\pa \xi}(\xi,\eta)\bigg|\leq\bigg|\fr{u+v}{2\xi}\bigg|+|b_1| +\int_{0}^\xi\bigg(\bigg|\fr{u_r+v_r}{2t}\bigg|+\bigg|\fr{\pa b_1}{\pa r}\bigg|\bigg)\cdot\bigg|\lambda_1\fr{\pa r_1}{\pa \eta}\bigg|\ {\rm d}t \nonumber \\[4pt]
\leq &\fr{1}{2}M\xi+K\xi(1+M\delta)^2+\int_{0}^\xi\bigg(Mt+Kt(1+M\delta)^3\bigg)\cdot Kt(1+M\delta) e^{K\delta^2(1+M\delta)^2} \ {\rm d}t \nonumber \\[4pt]
\leq&\bigg\{\fr{1}{2}M+K(1+M\delta)^2 +K\delta^2(1+M\delta)\bigg(M+K(1+M\delta)^3\bigg)e^{K\delta^2(1+M\delta)}\bigg\}\xi.
\end{align}
Similar arguments give the bounds for $V_\xi$ and $W_\xi$
\begin{align*}
\bigg|\fr{\pa V}{\pa \xi}(\xi,\eta)\bigg|
&\leq\bigg\{\fr{1}{2}M+K(1+M\delta)^2 +K\delta^2(1+M\delta)\bigg(M+K(1+M\delta)^3\bigg)e^{K\delta^2(1+M\delta)}\bigg\}\xi,   \\[4pt]
\bigg|\fr{\pa W}{\pa \xi}(\xi,\eta)\bigg|
&\leq \bigg\{K(1+M\delta)^2+K\delta^2(1+M\delta)^4e^{K\delta^2(1+M\delta)}\bigg\}\xi,
\end{align*}
which, along with \eqref{4.70} and the choices of $M$ and $\delta$,  leads to
\begin{align}\label{4.71}
\bigg|\fr{\pa U}{\pa \xi}(\xi,\eta)\bigg|+\bigg|\fr{\pa V}{\pa \xi}(\xi,\eta)\bigg| +\bigg|\fr{\pa W}{\pa \xi}(\xi,\eta)\bigg|\leq2M\xi.
\end{align}
\vspace{0.2cm}

\noindent {\bf Step 2. To demonstrate $|U_{tr}|+|V_{tr}|+|W_{tr}|\leq 2Mt$.}  Differentiating \eqref{4.56} with respect to $\eta$ arrives at
\begin{align}\label{4.72}
&\fr{\pa}{\pa \xi}\bigg(\fr{\pa U}{\pa \eta}(\xi,\eta)\bigg)=\bigg(\fr{u_r+v_r}{2\xi} +\fr{\pa b_1}{\pa r}\bigg)\cdot\fr{\pa r_1}{\pa\eta} \nonumber \\[4pt]
&\quad +\int_{0}^\xi\bigg\{\bigg(\fr{u_{rr}+v_{rr}}{2t} +\fr{\pa^2b_1}{\pa r^2}\bigg)\cdot\fr{\pa r_1}{\pa \eta}\cdot\fr{\pa r_1}{\pa \xi} +\bigg(\fr{u_{r}+v_{r}}{2t} +\fr{\pa b_1}{\pa r}\bigg)\cdot\fr{\pa^2 r_1}{\pa \xi\pa \eta}\bigg\}\ {\rm d}t,
\end{align}
where
$$
\fr{\pa^2 r_1}{\pa \xi\pa \eta}(t;\xi,\eta)=\exp\bigg(\int_{\xi}^t\fr{\pa \lambda_1}{\pa r}\ {\rm d}s\bigg)\cdot\bigg\{\int_{\xi}^t\fr{\pa^2 \lambda_1}{\pa r^2}\cdot\fr{\pa r_1}{\pa \xi}\ {\rm d}s -\fr{\pa\lambda_1}{\pa r}\bigg\}.
$$
Note that
\begin{align}\label{a5}
\bigg|\fr{\pa^2 r_1}{\pa \xi\pa \eta}\bigg|&\leq e^{K(1+M\delta)\delta^2}\bigg\{\int_{0}^\delta Kt(1+M\delta)^3\cdot Kt(1+M\delta) e^{K\delta^2(1+M\delta)^2}\ {\rm d}t+K\delta(1+M\delta)^2\bigg\} \nonumber \\[4pt]
& \leq K\delta(1+M\delta)^4e^{K\delta^2(1+M\delta)^2}.
\end{align}
Then we have
\begin{align}\label{4.73}
\bigg|\fr{\pa^2 U}{\pa\xi\pa\eta}\bigg|\leq&\bigg(\fr{1}{2}M\xi+K\xi(1+M\delta)^3\bigg)\cdot e^{K\delta^2(1+M\delta)^2} \nonumber \\[4pt]
&\ \ +\int_{0}^\xi\bigg\{\bigg(\fr{1}{2}Mt +Kt(1+M\delta)^4\bigg)\cdot Kt(1+M\delta) e^{K\delta^2(1+M\delta)^2} \nonumber \\[4pt]
&\qquad \qquad\ \ +\bigg(\fr{1}{2}Mt +Kt(1+M\delta)^3\bigg)\cdot K\delta(1+M\delta)^4e^{K\delta^2(1+M\delta)^2}\bigg\}\ {\rm d}\tau
\nonumber \\[5pt]
\leq& \bigg\{\bigg(\fr{1}{2}M+K(1+M\delta)^3\bigg) \nonumber \\[5pt] &\quad +K\delta^2(1+M\delta)^4\bigg(M+K(1+M\delta)^4\bigg) \bigg\}e^{K\delta^2(1+M\delta)^2}\xi.
\end{align}
The same bound is also taken for $V_{\xi\eta}$ by the symmetry. For the term $W_{\xi\eta}$, we can obtain
\begin{align*}
\bigg|\fr{\pa^2 W}{\pa\xi\pa\eta}\bigg|\leq K(1+M\delta)^3\bigg\{ 1+\delta^2(1+M\delta)^6\bigg\}e^{K\delta^2(1+M\delta)^2}\xi.
\end{align*}
Therefore, if $M$ and $\delta$ are chosen as above, one has
\begin{align}\label{4.74}
\bigg|\fr{\pa^2 U}{\pa\xi\pa\eta}\bigg|+\bigg|\fr{\pa^2 V}{\pa\xi\pa\eta}\bigg| +\bigg|\fr{\pa^2 W}{\pa\xi\pa\eta}\bigg|\leq 2M\xi.
\end{align}
\vspace{0.2cm}

\noindent {\bf  Step 3. To prove $|U_{tt}|+|V_{tt}|+|W_{tt}|\leq 7M$.}  We first estimate $\pa_t b_1$. Differentiating \eqref{4.13} with respect to $t$ gives
\begin{align}\label{4.75}
\fr{\pa b_1}{\pa t}(u,v,w,t,r)=&\bigg(-\fr{u_t+v_t}{2t}+\fr{u+v}{2t^2}\bigg)\fr{I_1}{F} -\bigg(\fr{u+v}{2t}+a_1\bigg)\fr{I_{1t}}{F}+\fr{2(1-t^2)t-t^3}{F\sqrt{1-t^2}}I_2 \nonumber \\[4pt]
&\ +\fr{t^2\sqrt{1-t^2}}{F}I_{2t} +\fr{I_3}{F} +\fr{tI_{3t}}{F} -\fr{b_1}{F}F'(t).
\end{align}
By a direct calculation, we obtain
\begin{align*}
I_{1t}=&\fr{(\kappa+1)(u_t+v_t+2a_1)}{v-Gw+\psi} -2t\fr{(\kappa+2-t^2)(v+a_0+a_1t)-(\kappa+1-t^2)G(w+H_0)}{v-Gw+\psi} \\[4pt] &-t^2\fr{(\kappa+2-t^2)(v_t+a_1)-(\kappa+1-t^2)(Gw_t+G'(w+H_0))}{v-Gw+\psi} \\[4pt]
&\ +2t^3\fr{v+a_0+a_1t-G(w+H_0)}{v-Gw+\psi}-\fr{I_1}{v-Gw+\psi}(v_t-Gw_t-G'w+\psi_t),
\end{align*}
\begin{align*}
I_{2t}=\fr{a'_1(v+a_0+a_1t)+(-a'_0+a'_1t)(v_t+a_1)}{v-Gw+\psi} -\fr{I_2(v_t-Gw_t-G'w+\psi_t)}{v-Gw+\psi},
\end{align*}
and
\begin{align*}
I_{3t}=&\fr{(\kappa+2-2t^2)(u_t+a_1)+(\kappa+1-t^2)(Gw_t+G'(w+H_0))}{v-Gw+\psi}(v+a_0+a_1t) \\[4pt]
&-2t\fr{2(u-a_0+a_1t)+G(w+H_0)}{v-Gw+\psi}(v+a_0+a_1t) \\[4pt] &\ +\fr{I_3(v_t+a_1)}{v+a_0+a_1t} -\fr{I_3(v_t-Gw_t-G'w+\psi_t)}{v-Gw+\psi}.
\end{align*}
We recall the definitions of $G, F$ and $\psi$ to calculate
$$
G'=\fr{-(\kappa+1)t}{(\kappa+1-t^2)^2}\bigg(\fr{1-t^2}{\kappa+1-t^2}\bigg)^{\fr{1-\kappa}{2\kappa}}, \ \ F'=-2t(\kappa+2-2t^2),\quad \psi_t=a_1-G'H_0.
$$
Thus we find, by using  \eqref{4.14} and the fact $(u,v,w)^T\in\mathcal{S}^{M}_\delta$,  that
\begin{align}\label{4.76}
|I_{1t}|&\leq K(1+M\delta) +K\delta(1+M\delta) +K\delta^2(1+M\delta) +K\delta^3(1+M\delta) +K\delta(1+M\delta)^2 \nonumber \\
&\leq K(1+M\delta)^2,
\end{align}
\begin{align}\label{4.77}
|I_{2t}|\leq K(1+M\delta)+K(1+M\delta)^2\leq K(1+M\delta)^2,
\end{align}
and
\begin{align}\label{4.78}
|I_{3t}|&\leq K(1+M\delta)^2+K\delta(1+M\delta)^2+K(1+M\delta)^2+K(1+M\delta)^3 \nonumber \\
&\leq K(1+M\delta)^3.
\end{align}
Inserting \eqref{4.76}-\eqref{4.78} into \eqref{4.75} and using \eqref{4.14} and \eqref{4.15} achieves
\begin{align}\label{4.79}
\bigg|\fr{\pa b_1}{\pa t}\bigg|\leq &K(M\delta+M)\cdot K\delta(1+M\delta) +K(1+M\delta)\cdot K(1+M\delta)^2 +K\delta\cdot K(1+M\delta) \nonumber \\
&+K\delta^2\cdot K(1+M\delta)^2 +K(1+M\delta)^2+ K\delta(1+M\delta)^3 +K\delta(1+M\delta)^2\cdot K\delta \nonumber \\
\leq & K(1+M\delta)^3.
\end{align}
Moreover, we compute
\begin{align*}
\fr{\pa \lambda_1}{\pa t}=&-\fr{t^2\sqrt{1-t^2}(v_t+a_1)}{F(v-Gw+\psi)} -\fr{(2t-3t^3)(v+a_0+a_1t)}{F\sqrt{1-t^2}(v-Gw+\psi)} \\[4pt] &\ -\fr{\lambda_1[F'(v-Gw+\psi)+F(v_t-Gw_t-G'w+\psi_t)]}{F(v-Gw+\psi)},
\end{align*}
from which   we have
$$
\bigg|\fr{\pa \lambda_1}{\pa t}\bigg|\leq K\delta(1+M\delta)^2.
$$
Then it follows by \eqref{4.57}, \eqref{a5} and the above inequality that
\begin{align}\label{4.80}
\bigg|\fr{\pa^2 r_1}{\pa \xi^2}\bigg|&\leq \bigg|\fr{\pa \lambda_1}{\pa\xi}\cdot\fr{\pa r_1}{\pa\eta}\bigg|+\bigg|\lambda_1\fr{\pa^2r_1}{\pa\xi\pa\eta}\bigg| \nonumber \\[4pt]
&\leq K\delta(1+M\delta)^2\cdot e^{K\delta^2(1+M\delta)^2} +K\delta(1+M\delta)\cdot K\delta(1+M\delta)^4e^{K\delta^2(1+M\delta)^2} \nonumber \\[4pt]
&\leq K\delta(1+M\delta)^3e^{K\delta^2(1+M\delta)^2}.
\end{align}

Now we differentiate \eqref{4.56} with respect to $\xi$ to obtain
\begin{align}\label{4.81}
&\fr{\pa}{\pa\xi}\bigg(\fr{\pa U}{\pa\xi}(\xi,\eta)\bigg)=\fr{u_\xi+v_\xi}{\xi}-\fr{u+v}{\xi^2}+2\fr{\pa b_1}{\pa\xi} \nonumber \\[4pt]
&\quad +\int_{0}^\xi\bigg\{\bigg(\fr{u_{rr}+v_{rr}}{2t} +\fr{\pa^2b_1}{\pa r^2}\bigg)\cdot\bigg(\fr{\pa r_1}{\pa\xi}\bigg)^2 +\bigg(\fr{u_{r}+v_{r}}{2t} +\fr{\pa b_1}{\pa r}\bigg)\cdot\fr{\pa^2 r_1}{\pa\xi^2}\bigg\}\ {\rm d}t,
\end{align}
which,  together with \eqref{4.43}, \eqref{4.49}, \eqref{4.57}, \eqref{4.71}, \eqref{4.79} and \eqref{4.80},  gives
\begin{align}\label{4.82}
&\bigg|\fr{\pa^2 U}{\pa \xi^2}\bigg|\leq \fr{|u_\xi|+|v_\xi|}{\xi}+\fr{|u|+|v|}{\xi^2}+2\bigg|\fr{\pa b_1}{\pa\xi}\bigg| \nonumber \\[4pt]
&\quad +\int_{0}^\xi\bigg\{\bigg(\fr{|u_{rr}|+|v_{rr}|}{2t} +\bigg|\fr{\pa^2b_1}{\pa r^2}\bigg|\bigg)\cdot\bigg|\fr{\pa r_1}{\pa\xi}\bigg|^2 +\bigg(\fr{|u_{r}|+|v_{r}|}{2t} +\bigg|\fr{\pa b_1}{\pa r}\bigg|\bigg)\cdot\bigg|\fr{\pa^2 r_1}{\pa\xi^2}\bigg|\bigg\}\ {\rm d}t \nonumber \\[4pt]
\leq& 2M+M+K(1+M\delta)^3+ \int_{0}^\xi\bigg\{\bigg(Mt +K(1+M\delta)^4t\bigg)\cdot K\delta^2(1+M\delta)^2e^{K\delta^2(1+M\delta)^2}  \nonumber \\[4pt]
 &\quad  +\bigg(Mt +K(1+M\delta)^3t\bigg)\cdot K\delta(1+M\delta)^3e^{K\delta^2(1+M\delta)^2}\bigg\}\ {\rm d}t \nonumber \\[4pt]
\leq& 3M+K(1+M\delta)^3+K\delta^2(1+M\delta)^4e^{K\delta^2(1+M\delta)^2}.
\end{align}
Similar arguments apply  for $V$ and $W$ to show
\begin{align*}
\bigg|\fr{\pa^2 V}{\pa \xi^2}\bigg|&\leq 3M+K(1+M\delta)^3+K\delta^2(1+M\delta)^4e^{K\delta^2(1+M\delta)^2}, \\[4pt]
\bigg|\fr{\pa^2 W}{\pa \xi^2}\bigg|&\leq K(1+M\delta)^3+K\delta^2(1+M\delta)^4e^{K\delta^2(1+M\delta)^2},
\end{align*}
which, along with \eqref{4.82},  yields
\begin{align}\label{4.83}
\bigg|\fr{\pa^2 U}{\pa \xi^2}\bigg| + \bigg|\fr{\pa^2 V}{\pa \xi^2}\bigg| +\bigg|\fr{\pa^2 W}{\pa \xi^2}\bigg| &\leq 6M +K(1+M\delta)^3+K\delta^2(1+M\delta)^4e^{K\delta^2(1+M\delta)^2} \nonumber \\
&\leq 7M,
\end{align}
if $M$ and $\delta$ are chosen as above, i.e., $M\geq64K$ and $\delta\leq1/M$. We combine \eqref{4.71}, \eqref{4.74} and \eqref{4.83} and employ Lemma \ref{lem1} to finish the proof of Lemma \ref{lem2}.
\end{proof}

\section{Existence of solutions in the original $(x,y)$ plane}\label{s5}

In view of \eqref{4.1}, we see that Problem \ref{prob2} and Problem \ref{prob3} are equivalent, which says by Theorem \ref{thm3} that the boundary value problem \eqref{3.9} with \eqref{3.13}-\eqref{3.14} has a local classical solution. This section is devoted to constructing a smooth solution for Problem \ref{p1} by converting the solution in the $(t,r)$-plane to that in the original  $(x,y)$-plane.

Thanks to the coordinate transformation \eqref{3.3}, it is easy to derive
\begin{align}\label{5.1}
\left\{
\begin{array}{l}
\fr{\pa x}{\pa t}=\fr{\theta_y}{J},\\[4pt]
\fr{\pa y}{\pa t}=-\fr{\theta_x}{J},
\end{array}
\right. \quad
\left\{
\begin{array}{l}
\fr{\pa x}{\pa r}=\fr{\sin\omega\omega_y}{J},\\[4pt]
\fr{\pa y}{\pa r}=-\fr{\sin\omega\omega_x}{J},
\end{array}
\right.
\end{align}
where $J$ is given in \eqref{3.5}. We use \eqref{2.6}, \eqref{2.15} and \eqref{a3} to obtain
\begin{align}\label{5a1}
\theta_x&=\sin\beta\bar{\pa}^+\Xi+\sin\alpha\bar{\pa}^-\Xi \nonumber \\
&=(t\sin r-\sqrt{1-t^2}\cos r)\overline{U}+(t\sin r+\sqrt{1-t^2}\cos r)\overline{V},\nonumber \\
\theta_y&=-\cos\beta\bar{\pa}^+\Xi-\cos\alpha\bar{\pa}^-\Xi \nonumber\\
&=-(t\cos r+\sqrt{1-t^2}\sin r)\overline{U}-(t\cos r-\sqrt{1-t^2}\sin r)\overline{V},
\end{align}
and
\begin{align}\label{5a2}
\omega_x&=-\fr{\kappa+\sin^2\omega}{\cos^2\omega}[\sin\beta(\bar{\pa}^+\Xi+GH)-\sin\alpha(\bar{\pa}^-\Xi-GH)]
 \nonumber \\[4pt]
 &=-\fr{\kappa+1-t^2}{t^2}\bigg\{(t\sin r-\sqrt{1-t^2}\cos r)\overline{U} -(t\sin r+\sqrt{1-t^2}\cos r)\overline{V}+2t\sin rG\overline{H}\bigg\}, \nonumber \\[4pt]
\omega_y&=\fr{\kappa+\sin^2\omega}{\cos^2\omega}[\cos\beta(\bar{\pa}^+\Xi+GH)-\cos\alpha(\bar{\pa}^-\Xi-GH)] \nonumber
\\[4pt]
 &=\fr{\kappa+1-t^2}{t^2}\bigg\{(t\cos r+\sqrt{1-t^2}\sin r)\overline{U} -(t\cos r-\sqrt{1-t^2}\sin r)\overline{V}+2t\cos rG\overline{H}\bigg\}.
\end{align}
Then we have
\begin{align}
x_t&=-\fr{(t\cos r+\sqrt{1-t^2}\sin r)\overline{U}(t,r)+(t\cos r-\sqrt{1-t^2}\sin r)\overline{V}(t,r)}{2F(t)\{2\overline{U}(t,r)\overline{V}(t,r) +G(t)\overline{H}(t,r)[\overline{V}(t,r)-\overline{U}(t,r)]\}}t, \label{5.2} \\[4pt]
y_t&=-\fr{(t\sin r-\sqrt{1-t^2}\cos r)\overline{U}(t,r)+(t\sin r+\sqrt{1-t^2}\cos r)\overline{V}(t,r)}{2F(t)\{2\overline{U}(t,r)\overline{V}(t,r) +G(t)\overline{H}(t,r)[\overline{V}(t,r)-\overline{U}(t,r)]\}}t, \label{5.3}
\end{align}
and
\begin{align}
x_r&=\fr{(t\cos r+\sqrt{1-t^2}\sin r)\overline{U}-(t\cos r-\sqrt{1-t^2}\sin r)\overline{V}+2t\cos rG\overline{H}}{2t\sqrt{1-t^2}[2\overline{U}\cdot\overline{V}+G\overline{H}(\overline{V}-\overline{U})]},\label{5.4} \\[4pt]
y_r&=\fr{(t\sin r-\sqrt{1-t^2}\cos r)\overline{U}-(t\sin r+\sqrt{1-t^2}\cos r)\overline{V}+2t\sin rG\overline{H}}{2t\sqrt{1-t^2}[2\overline{U}\cdot\overline{V}+G\overline{H}(\overline{V}-\overline{U})]}.\label{5.5}
\end{align}
Thus we construct the functions $x=x(t,r)$ and $y=y(t,r)$ by solving the following equations
\begin{align*}
\left\{
\begin{array}{l}
\fr{\pa x}{\pa t}=-\fr{(t\cos r+\sqrt{1-t^2}\sin r)\overline{U}(t,r)+(t\cos r-\sqrt{1-t^2}\sin r)\overline{V}(t,r)}{2F(t)\{2\overline{U}(t,r)\overline{V}(t,r) +G(t)\overline{H}(t,r)[\overline{V}(t,r)-\overline{U}(t,r)]\}}t, \\[4pt]
x(0,r)=\hat{\theta}^{-1}(r),
\end{array}
\right.
\end{align*}
and
\begin{align*}
\left\{
\begin{array}{l}
\fr{\pa y}{\pa t}=-\fr{(t\sin r-\sqrt{1-t^2}\cos r)\overline{U}(t,r)+(t\sin r+\sqrt{1-t^2}\cos r)\overline{V}(t,r)}{2F(t)\{2\overline{U}(t,r)\overline{V}(t,r) +G(t)\overline{H}(t,r)[\overline{V}(t,r)-\overline{U}(t,r)]\}}t, \\[4pt]
y(0,r)=\varphi(\hat{\theta}^{-1}(r)),
\end{array}
\right.
\end{align*}
where $\hat{\theta}^{-1}$ denotes the inverse of $\hat{\theta}$, which exists by the strictly monotonic assumption of $\hat{\theta}$. Furthermore, the Jacobian of the map $(t,r)\mapsto(x,y)$ is
$$
j:=\fr{\pa(x,y)}{\pa(t,r)}=x_ty_r-x_ry_t =\fr{t}{2F(t)\{2\overline{U}(t,r)\overline{V}(t,r)+G\overline{H}(t,r)[\overline{V}(t,r)-\overline{U}(t,r)]\}},
$$
which is strictly less than zero when $t\in(0,\delta]$. Therefore, the map $(t,r)\mapsto(x,y)$ is an one-to-one mapping for $t\in(0,\delta]$. Thus we have the functions $t=t(x,y)$ and $r=r(x,y)$, and then define by \eqref{3.3}
\begin{align}\label{5.6}
\begin{array}{l}
\theta=r(x,y),\quad \omega=\arccos t(x,y), \quad H=\overline{H}(t(x,y),r(x,y)),\\
\alpha=r(x,y)+\arccos t(x,y),\quad \beta=r(x,y)-\arccos t(x,y).
\end{array}
\end{align}
We next check that the functions defined in \eqref{5.6} satisfy system \eqref{2.12}. Using \eqref{5.2}-\eqref{5.5}, one directly calculates
\begin{align}\label{5.7}
\bar{\pa}^+\theta+\fr{\cos^2\omega}{\sin^2\omega+\kappa}\bar{\pa}^+\omega &=\cos\alpha r_x+\sin\alpha r_y-\fr{t^2}{\sqrt{1-t^2}(\kappa+1-t^2)}(\cos\alpha t_x+\sin\alpha t_y) \nonumber \\[4pt]
&=\fr{1}{j}\bigg\{-\cos\alpha y_t+\sin\alpha x_t-\fr{t^2\sqrt{1-t^2}}{F(t)}(\cos\alpha y_r-\sin\alpha x_r)\bigg\} \nonumber \\[4pt]
&=\fr{1}{j}\cdot\fr{t^2}{F(t)[2\overline{U}\cdot\overline{V}+G\overline{H}(\overline{V} -\overline{U})]}(\sin\alpha\cos r-\cos\alpha\sin r)G\overline{H}\nonumber \\[4pt]
&=2t\sin(\alpha-r)G\overline{H}=\sin(2\omega)GH,
\end{align}
which means that the first equation of \eqref{2.12} holds. The second equation of \eqref{2.12} can be checked  analogously. For the third one, we find from \eqref{3.9} that
\begin{align}\label{5.8}
\bar{\pa}^0H&=\cos r(\overline{H}_tt_x+\overline{H}_rr_x)+\sin r(\overline{H}_tt_y+\overline{H}_rr_y) \nonumber \\[4pt]
&=\fr{1}{j}\bigg\{(\cos ry_r-\sin rx_r)\overline{H}_t-(\cos ry_t-\sin rx_t)\overline{H}_r\bigg\}  \nonumber \\[4pt]
&=\fr{1}{j}\bigg\{\fr{\overline{U}+\overline{V}}{2t[2\overline{U}\cdot\overline{V}+G\overline{H}(\overline{V} -\overline{U})]}\overline{H}_t+\fr{t\sqrt{1-t^2}(\overline{U}- \overline{V})}{2F[2\overline{U}\cdot\overline{V}+G\overline{H}(\overline{V} -\overline{U})]}\overline{H}_r\bigg\}\nonumber \\[4pt]
&=\fr{F(\overline{U}+\overline{V})}{t^2}\bigg\{\overline{H}_t +\fr{t^2\sqrt{1-t^2}(\overline{U}-\overline{V})}{F(\overline{U}+\overline{V})}\overline{H}_r\bigg\}=0.
\end{align}
Thus the proof of Theorem \ref{thm1} is completed.
\vspace{0.2cm}

In addition,  it follows by \eqref{5.6} that the fourth equation of \eqref{2.12} now is a linear equation, from which and the boundary condition $S(x,\varphi(x))=\hat{S}(x)$ we get the function $S=S(x,y)$. Similarly, we integrate the equation \eqref{a2} with $B(x,\varphi(x))=\hat{B}(x)$ to give the function $B=B(x,y)$. Here $\hat{S}(x)$ and $\hat{B}(x)$ are defined in \eqref{a6}. Hence we obtain the classical solution $(\theta, \omega, S, B)(x,y)$ of system \eqref{2.7} near the sonic curve $\Gamma$. Moreover, by the functions $(\theta, \omega, S, B)(x,y)$, one can define
\begin{align}\label{5.9}
\begin{array}{l}
c=\sqrt{\fr{2\kappa\sin^2\omega(x,y) B(x,y)}{\kappa+\sin^2\omega(x,y)}},\quad u=c(x,y)\fr{\cos\theta(x,y)}{\sin\omega(x,y)},\quad v=c(x,y)\fr{\sin\theta(x,y)}{\sin\omega(x,y)},\\[4pt]
\rho=\bigg(\fr{2\kappa B(x,y)\sin^2\omega(x,y)}{\gamma [\kappa +\sin^2\omega(x,y)]S(x,y)}\bigg)^{\fr{1}{\gamma-1}},\quad p=S(x,y)\bigg(\fr{2\kappa B(x,y)\sin^2\omega(x,y)}{\gamma [\kappa +\sin^2\omega(x,y)]S(x,y)}\bigg)^{\fr{\gamma}{\gamma-1}}.
\end{array}
\end{align}
It is not difficult to check that the functions $(\rho, u,v,p)(x,y)$ defined as above satisfy system \eqref{2.2} subject to the boundary condition \eqref{1.3}. Thus we finish the proof of Theorem \ref{thm2}.

\section*{Acknowledgements}

Yanbo Hu was supported by the Zhejiang Provincial Natural Science Foundation (No. LY17A010019) and National Science Foundation of China (No: 11301128). Jiequan Li was supported by the Natural Science Foundation of China (Nos: 11771054, 91852207) and Foundation of LCP.

\section*{Appendices}
\appendix

\section{The Tricomi equation}\label{sa}

In order to illustrate the methodology of this paper, we take the classical Tricomi equation and propose a similar problem.
The Tricomi equation is a second-order partial
differential equation of mixed hyperbolic-elliptic
type with the form,
\begin{align}\label{s1.1}
yu_{xx}+u_{yy}=0.
\end{align}
Equation \eqref{s1.1} is hyperbolic in the half plane $y<0$,
elliptic in the half plane $y>0$, and degenerates on the line $y=0$.
The characteristic equation of \eqref{s1.1} in $y<0$ is
$$
dx^2+ydy^2=0,
$$
which gives the explicit expression of characteristic curves,
$$
x\pm\fr{2}{3}(-y)^{\fr{3}{2}}=C
$$
for any constant $C$. These two families of characteristics
coincide  and  form cusps perpendicularly to the line $y=0$. See Figure 2.

\vspace{0.2cm}

The solution of \eqref{s1.1} is now well-understood, e.g., in \cite{Bi}. Below we just use our method for a hyperbolic degenerate problem with prescribed data on the degenerate line $y=0$.

\begin{figure}[htbp]
\begin{center}
\includegraphics[scale=0.58]{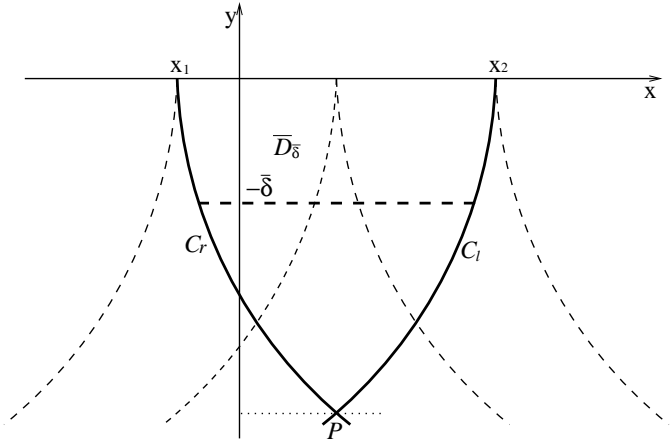}
\caption{\footnotesize Characteristics of Tricomi equation and region $\overline{D}_{\bar{\delta}}$.}
\end{center}
\end{figure}

\subsection{A hyperbolic degenerate problem}

We prescribe the boundary data on $y=0$
\begin{align}\label{s1.2}
u(x,0)=u_0(x),\quad u_y(x,0)=u_1(x),\ \ x\in[x_1,x_2].
\end{align}
The rightward characteristics starting from point $(x_1, 0)$ (denoted by $C_r$) and the leftward characteristics starting from point $(x_2, 0)$ (denoted by $C_l$) are, respectively, $x=x_1+\fr{2}{3}(-y)^{\fr{3}{2}}$ and $x=x_2-\fr{2}{3}(-y)^{\fr{3}{2}}$,
which intersect at the point $P(\fr{x_2-x_1}{2},-\sqrt[3]{\fr{9(x_2-x_1)^2}{16}})$.  Let $\bar{\delta}\leq\sqrt[3]{\fr{9(x_2-x_1)^2}{16}}$ be a positive constant. Denote
$$
\overline{D}_{\bar{\delta}}=\{(y,x)|\ -\bar{\delta}\leq y\leq0,\ x_1+\fr{2}{3}(-y)^{\fr{3}{2}}\leq x\leq x_2-\fr{2}{3}(-y)^{\fr{3}{2}}\}.
$$
See Figure 2. The basic  problem is
\begin{prob}\label{pT1a}
we look for a classical solution for The Tricomi equation \eqref{s1.1}
with the boundary condition \eqref{s1.2} in the region $\overline{D}_{\bar{\delta}}$ for some constant $\bar{\delta}>0$.
\end{prob}

In the context of strictly hyperbolic problems, this is the typical Goursat problem and the solution can be constructed with the standard method of characteristics \cite{Li-Yu}.  However, due to the degeneracy at $y=0$, the characteristic method  cannot be applied directly \cite{Han-Hong}. The method in the present study is ``nonlinear" in the sense that it can deal with nonlinear problems.   The theorem is stated as follows.

\begin{thm}\label{thmT2}
Assume $u_0(x)\in C^4([x_1,x_2])$ and $u_1(x)\in C^3([x_1,x_2])$. Then there exists a positive constant $\bar{\delta}\leq\sqrt[3]{\fr{9(x_2-x_1)^2}{16}}$ such that the boundary
problem \eqref{s1.1} \eqref{s1.2} has a classical solution in the region $\overline{D}_{\bar{\delta}}$.
\end{thm}
\begin{rem}\label{remT1}
Since the Tricomi equation is a linear equation, it is not difficult to extend the solution in Theorem \ref{thmT2} to the whole region bounded by $x=x_1+\fr{2}{3}(-y)^{\fr{3}{2}}$, $x=x_2-\fr{2}{3}(-y)^{\fr{3}{2}}$ and $y=0$.
\end{rem}

To establish Theorem \ref{thmT2}, we rewrite the second-order equation \eqref{s1.1} to a hyperbolic system.
Let
$$
\overline{R}=u_y+\sqrt{-y}u_x,\quad \overline{S}=u_y-\sqrt{-y}u_x.
$$
Then, for smooth solutions, equation \eqref{s1.1} is equivalent to
\begin{align}\label{s1.3}
\left\{
\begin{array}{l}
 \overline{R}_y-\sqrt{-y}\overline{R}_x=\fr{\overline{R}-\overline{S}}{4y}, \\
 \overline{S}_y+\sqrt{-y}\overline{S}_x=\fr{\overline{S}-\overline{R}}{4y}, \\
 u_y+\sqrt{-y}u_x=\overline{R}.
\end{array}
\right.
\end{align}
Note that $\overline{R}-\overline{S}=2\sqrt{-y}u_x$ vanishes as the rate $\sqrt{-y}$, which means that the term $(\overline{R}-\overline{S})/y$ in system \eqref{s1.3} blows up in the order of $(-y)^{-\fr{1}{2}}$ when approaching the line $y=0$. Observing this singularity, we introduce the following transformation
\begin{align}\label{s1.4}
t=\sqrt{-y},\quad x=x,
\end{align}
which is an one-to-one mapping for $y\leq0$. Denote $\widetilde{R}(x,t)=\overline{R}(x,y)$, $\widetilde{S}(x,t)=\overline{S}(x,y)$ and $\widetilde{u}(x,t)=u(x,y)$. Then system \eqref{s1.3} can be rewritten as
\begin{align}\label{s1.5}
\left\{
\begin{array}{l}
 \widetilde{R}_t+2t^2\widetilde{R}_x=\fr{\widetilde{R}-\widetilde{S}}{2t}, \\
 \widetilde{S}_t-2t^2\widetilde{S}_x=\fr{\widetilde{S}-\widetilde{R}}{2t}, \\
 \widetilde{u}_t-2t^2\widetilde{u}_x=-2t\widetilde{R}.
\end{array}
\right.
\end{align}
Corresponding to \eqref{s1.2}, we have the boundary conditions for system \eqref{s1.5} as follows
\begin{align}\label{s1.6}
\begin{array}{l}
\quad (\widetilde{R},\widetilde{S},\widetilde{u})(x,0)=(u_1,u_1,u_0)(x),\\ (\widetilde{R}_t,\widetilde{S}_t,\widetilde{u}_t)(x,0)=(u_0',-u_0',0)(x),
\end{array}\ \ x\in[x_1,x_2].
\end{align}
To  better understand the singularity, we make the Taylor expansion for $(\widetilde{R},\widetilde{S},\widetilde{u})$ and introduce the following higher order error terms for the variables $(\widetilde{R},\widetilde{S},\widetilde{u})$
\begin{align}\label{s1.7}
\left\{
\begin{array}{l}
 R=\widetilde{R}(x,t)-u_1(x)-u_0'(x)t, \\
 S=\widetilde{S}(x,t)-u_1(x)+u_0'(x)t, \\
 W=\widetilde{u}(x,t)-u_1(x).
\end{array}
\right.
\end{align}
Then one has the system for variables $(R,S,W)$
\begin{align}\label{s1.8}
\left\{
\begin{array}{l}
 R_t+2t^2R_x=\fr{R-S}{2t}-2t^2(u_1'+u_0''t), \\
 S_t-2t^2S_x=\fr{S-R}{2t}+2t^2(u_1'-u_0''t), \\
 W_t-2t^2W_x=-2t(R+u_1),
\end{array}
\right.
\end{align}
with the following boundary conditions
\begin{align}\label{s1.9}
\begin{array}{l}
\quad (R,S,W)(x,0)=(0,0,0),\\
(R_t,S_t,W_t)(x,0)=(0,0,0),
\end{array}
\ \ x\in[x_1,x_2].
\end{align}

For system \eqref{s1.8}, the positive characteristic curve from $(x_1,0)$ and negative characteristic curve from $(x_2,0)$ are, respectively, $x=x_1+\fr{2}{3}t^3$ and $x=x_2-\fr{2}{3}t^3$, which intersect at point $(\fr{x_2-x_1}{2}, \sqrt[3]{\fr{3(x_2-x_1)}{4}})$. Let $\tilde{\delta}\leq\sqrt[3]{\fr{3(x_2-x_1)}{4}}$ be a small positive constant. Then we define a region in the plane $(x,t)$
$$
\widetilde{D}_{\tilde{\delta}}=\{(x,t)|\ 0\leq t\leq\tilde{\delta},\ x_1+\fr{2}{3}t^3\leq x\leq x_2-\fr{2}{3}t^3\}.
$$
Hence, Problem \ref{pT1a}  is reformulated to the following new problem in the $(x,t)$ plane. and the corresponding existence theorem holds in parallel.

\begin{prob}\label{pT1}
we look for a classical solution for system \eqref{s1.8}
with the boundary condition \eqref{s1.9} in the region $\widetilde{D}_{\tilde{\delta}}$ for some constant $\tilde{\delta}>0$.
\end{prob}

\begin{thm}\label{thmT1}
Assume $u_0(x)\in C^4([x_1,x_2])$ and $u_1(x)\in C^3([x_1,x_2])$, then there exists positive constants $\tilde{\delta}\leq\sqrt[3]{\fr{3(x_2-x_1)}{4}}$ such that the boundary
problem \eqref{s1.8} \eqref{s1.9} has a classical solution in the region $\widetilde{D}_{\tilde{\delta}}$.
\end{thm}
Since the map $(t,x)\rightarrow(y,x)$ is an one-to-one mapping for $y\leq0$, then Theorem \ref{thmT2} follows directly from Theorem \ref{thmT1}.

To prove Theorem \ref{thmT1}, we need to consider the problem in a refined class of functions due to the fact that the variables $(R, S, W)$ have very small magnitude near the line $t=0$. Let $\mathcal{S}_{\tilde{\delta}}^{\tilde{M}}$ be a function class which incorporates all continuously differentiable vector functions $\mathbf{F}=(f_1,f_2,f_3)^T: \widetilde{D}_{\tilde{\delta}}\rightarrow\mathbb{R}^3$ satisfying the following properties:
\begin{align*}
&{\rm (P1)}\ \mathbf{F}(x,0)=\mathbf{F}_t(x,0)=0, \rule{10cm}{0ex} \\[3pt]
&{\rm (P2)}\ \bigg\|\fr{\mathbf{F}(x,t)}{t^2}\bigg\|_\infty\leq \tilde{M}, \rule{10cm}{0ex}
\\[3pt]
&{\rm (P3)}\ \bigg\|\fr{\pa_x\mathbf{F}(x,t)}{t^2}\bigg\|_\infty\leq \tilde{M}, \rule{10cm}{0ex}
\\[3pt]
&{\rm (P4)}\ \pa_x\mathbf{F}(x,t)\ {\rm is\ Lipschitz\ continuous\ with\ respect\ to\ x\ and}\ \bigg\|\fr{\pa_{xx}\mathbf{F}(x,t)}{t^2}\bigg\|_\infty\leq \tilde{M},
\end{align*}
where $\tilde{M}$ is a positive constant and $\|\cdot\|_\infty$ denotes the supremum norm on the domain $\widetilde{D}_{\tilde{\delta}}$. Denote $\mathcal{W}_{\tilde{\delta}}^{\tilde{M}}$ the function class containing only continuous functions on $\widetilde{D}_{\tilde{\delta}}$ which satisfy the first two conditions ${\rm (P1)}$ and ${\rm (P2)}$. Obviously, $\mathcal{S}_{\tilde{\delta}}^{\tilde{M}}$ is a subset of $\mathcal{W}_{\tilde{\delta}}^{\tilde{M}}$ and both $\mathcal{S}_{\tilde{\delta}}^{\tilde{M}}$ and $\mathcal{W}_{\tilde{\delta}}^{\tilde{M}}$ are subsets of $C^0(\widetilde{D}_{\tilde{\delta}}, \mathbb{R}^3)$. For any elements $\mathbf{F}=(f_1,f_2,f_3)^T, \mathbf{G}=(g_1,g_2,g_3)^T$ in the set $\mathcal{W}_{\tilde{\delta}}^{\tilde{M}}$, we define the weighted metric as follows:
\begin{align}\label{s1.10}
d(\mathbf{F}, \mathbf{G}):=\bigg\|\fr{f_1-g_1}{t^2}\bigg\|_\infty+ \bigg\|\fr{f_2-g_2}{t^2}\bigg\|_\infty +\bigg\|\fr{f_3-g_3}{t^2}\bigg\|_\infty.
\end{align}
One can check that $(\mathcal{W}_{\tilde{\delta}}^{\tilde{M}},d)$ is a complete metric space, while the subset $(\mathcal{S}_{\tilde{\delta}}^{\tilde{M}},d)$ is not closed in the space $(\mathcal{W}_{\tilde{\delta}}^{\tilde{M}},d)$.

\subsection{The proof of Theorem \ref{thmT1}}

We use the fixed point method to prove Theorem \ref{thmT1} and divide the proof into three steps.
In Step 1, we construct an integration iteration mapping
in the function class $\mathcal{S}_{\tilde{\delta}}^{\tilde{M}}$. In Step 2, we demonstrate the mapping is a
contraction for some constants $\tilde{\delta}$ and $\tilde{M}$, which implies that the iteration sequence converge to a vector function in the limit. In Step 3, we show that this limit vector function also belongs to $\mathcal{S}_{\tilde{\delta}}^{\tilde{M}}$.
\vspace{0.2cm}

\noindent {\bf Step 1 (The iteration mapping).} Denote
$$
\fr{\rm d}{{\rm d}_\pm t}=\pa_t\pm 2t^2\pa_x.
$$
Let vector function $(r,s,w)(x,t)\in\mathcal{S}_{\tilde{\delta}}^{\tilde{M}}$. Then we consider the following system
\begin{align}\label{s1.11}
\left\{
\begin{array}{l}
\fr{\rm d}{{\rm d}_+ t}R=\fr{r-s}{2t}-2t^2(u_1'+u_0''t),\\[4pt]
\fr{\rm d}{{\rm d}_- t}S=\fr{s-r}{2t}+2t^2(u_1'-u_0''t), \\[4pt]
\fr{\rm d}{{\rm d}_- t}W=-2t(r+u_1).
\end{array}
\right.
\end{align}
The integral form of \eqref{s1.11} is
\begin{align}\label{s1.12}
\left\{
\begin{array}{l}
R(\eta,\xi)=\displaystyle\int_{0}^\xi\bigg(\fr{r(x_+(t),t)-s(x_+(t),t)}{2t}-2t^2(u_1'(x_+(t))+u_0''(x_+(t))t)\bigg)\ {\rm d}t,\\[10pt]
S(\eta,\xi)=\displaystyle\int_{0}^\xi\bigg(\fr{s(x_-(t),t)-r(x_-(t),t)}{2t}+2t^2(u_1'(x_-(t))-u_0''(x_-(t))t)\bigg)\ {\rm d}t, \\[10pt]
W(\eta,\xi)=-2\displaystyle\int_{0}^\xi t\{r(x_-(t),t)+u_1(x_-(t))\}\ {\rm d}t,
\end{array}
\right.
\end{align}
where
$$
x_\pm(t)=\pm\fr{2}{3}t^3+\eta\mp\fr{2}{3}\xi^3.
$$
Then \eqref{s1.12} determines an iteration mapping $\widetilde{\mathcal{T}}$:
\begin{align}\label{s1.13}
\widetilde{\mathcal{T}}
\left(
\begin{array}{l}
r \\
s \\
w
\end{array}
\right)=
\left(
\begin{array}{l}
R \\
S \\
W
\end{array}
\right),
\end{align}
and the existence of classical solutions for the boundary problem \eqref{s1.8} \eqref{s1.9} is
equivalent to the existence of fixed point for the mapping $\widetilde{\mathcal{T}}$ in the function class $\mathcal{S}_{\tilde{\delta}}^{\tilde{M}}$ for some constants $\tilde{\delta}$ and $\tilde{M}$.
\vspace{0.2cm}

\noindent {\bf Step 2 (Properties of the mapping).} For the mapping $\widetilde{\mathcal{T}}$, we have the following lemma.
\begin{lem}\label{lemT1}
Let the assumptions in Theorem \ref{thmT1} hold. Then there exists positive constants $\tilde{\delta}\leq\sqrt[3]{\fr{3(x_2-x_1)}{4}}, \tilde{M}$ and $0<\tilde{\nu}<1$ depending only on the $C^4$ norm of $u_0$ and $C^3$ norm of $u_1$ such that

\noindent (1) $\widetilde{\mathcal{T}}$ maps $\mathcal{S}^{\tilde{M}}_{\tilde{\delta}}$ into $\mathcal{S}_{\tilde{\delta}}^{\tilde{M}}$;

\noindent (2) For any vector functions $\mathbf{F}, \widehat{\mathbf{F}}$ in $\mathcal{S}_{\tilde{\delta}}^{\tilde{M}}$, there holds
\begin{align}\label{s1.14}
d\bigg(\widetilde{\mathcal{T}}(\mathbf{F}),\widetilde{\mathcal{T}}(\widehat{\mathbf{F}})\bigg)\leq\tilde{\nu} d(\mathbf{F},\widehat{\mathbf{F}}).
\end{align}
\end{lem}
\begin{proof}
Let $\mathbf{F}=(r,s,w)^T$ and $\hat{\mathbf{F}}=(\hat{r},\hat{s}, \hat{w})^T$ be two elements in $\mathcal{S}_{\tilde{\delta}}^{\tilde{M}}$, where the constants $\tilde{\delta}$ and $\tilde{M}$ will be determined later. Denote $\mathbf{G}=\overline{\mathcal{T}}(\mathbf{F})=(R,S,W)^T$ and $\widehat{\mathbf{G}}=\overline{\mathcal{T}}(\hat{\mathbf{F}})=(\hat{R},\hat{S},\hat{W})^T$. By the properties of $\mathcal{S}_{\tilde{\delta}}^{\tilde{M}}$, one has
\begin{align}\label{s1.15}
\begin{array}{r}
|r(x,t)|+|s(x,t)|+|w(x,t)|\leq \tilde{M}t^2,\\
|r_x(x,t)|+|s_x(x,t)|+|w_x(x,t)|\leq \tilde{M}t^2,\\
|r_{xx}(x,t)|+|s_{xx}(x,t)|+|w_{xx}(x,t)|\leq \tilde{M}t^2.
\end{array}
\end{align}
Denote
$$
K=1+\|u_1\|_{C^3([x_1,x_2])}+\sqrt[3]{\fr{3(x_2-x_1)}{4}}\|u_0\|_{C^4([x_1,x_2])}.
$$
It follows that
\begin{align*}
|R(\eta,\xi)|&\leq\int_{0}^\xi\bigg(\fr{\tilde{M}}{2}t+2Kt^2\bigg)\ {\rm d}t\leq \bigg(\fr{\tilde{M}}{4}+\fr{2K}{3}\tilde{\delta}\bigg)\xi^2,\\[4pt]
|S(\eta,\xi)|&\leq\int_{0}^\xi\bigg(\fr{\tilde{M}}{2}t+2Kt^2\bigg)\ {\rm d}t\leq \bigg(\fr{\tilde{M}}{4}+\fr{2K}{3}\tilde{\delta}\bigg)\xi^2, \\[4pt]
|W(\eta,\xi)|&\leq2\int_{0}^\xi t(\tilde{M}t^2+K)\ {\rm d}t\leq (\tilde{M}\tilde{\delta}^2+K)\xi^2.
\end{align*}
Hence we have
\begin{align}\label{s1.16}
\fr{|R(\eta,\xi)|+|S(\eta,\xi)|+|W(\eta,\xi)|}{\xi^2}\leq\bigg(\fr{1}{2}+\tilde{\delta}^2\bigg)\tilde{M} +\bigg(1+\fr{4}{3}\tilde{\delta}\bigg)K.
\end{align}

Noting the fact $\fr{\pa x_\pm}{\pa \eta}=1$, we differentiate system \eqref{s1.12} with respect to $\eta$ to get
\begin{align*}
\fr{\pa R}{\pa \eta}(\eta,\xi) &=\displaystyle\int_{0}^\xi\bigg(\fr{r_x(x_+(t),t)-s_x(x_+(t),t)}{2t}-2t^2(u_1''(x_+(t))+u_0'''(x_+(t))t)\bigg)\ {\rm d}t,\\[4pt]
\fr{\pa S}{\pa \eta}(\eta,\xi) &=\displaystyle\int_{0}^\xi\bigg(\fr{s_x(x_-(t),t)-r_x(x_-(t),t)}{2t}+2t^2(u_1''(x_-(t))-u_0'''(x_-(t))t)\bigg)\ {\rm d}t, \\[4pt]
\fr{\pa W}{\pa \eta}(\eta,\xi)&=-2\displaystyle\int_{0}^\xi t\{r_x(x_-(t),t)+u_1'(x_-(t))\}\ {\rm d}t,
\end{align*}
subsequently,
\begin{align*}
\fr{\pa^2 R}{\pa \eta^2}(\eta,\xi) &=\displaystyle\int_{0}^\xi\bigg(\fr{r_{xx}(x_+(t),t)-s_{xx}(x_+(t),t)}{2t} -2t^2(u_1'''(x_+(t))+u_0''''(x_+(t))t)\bigg)\ {\rm d}t,\\[4pt]
\fr{\pa^2 S}{\pa \eta^2}(\eta,\xi) &=\displaystyle\int_{0}^\xi\bigg(\fr{s_{xx}(x_-(t),t)-r_{xx}(x_-(t),t)}{2t} +2t^2(u_1'''(x_-(t))-u_0''''(x_-(t))t)\bigg)\ {\rm d}t, \\[4pt]
\fr{\pa^2 W}{\pa \eta^2}(\eta,\xi) &=-2\displaystyle\int_{0}^\xi t\{r_{xx}(x_-(t),t)+u_1''(x_-(t))\}\ {\rm d}t.
\end{align*}
Therefore we obtain
\begin{align}\label{s1.17}
\fr{|R_\eta(\eta,\xi)|+|S_\eta(\eta,\xi)|+|W_\eta(\eta,\xi)|}{\xi^2}&\leq\bigg(\fr{1}{2} +\tilde{\delta}^2\bigg)\tilde{M} +\bigg(1+\fr{4}{3}\tilde{\delta}\bigg)K, \nonumber \\[4pt]
\fr{|R_{\eta\eta}(\eta,\xi)|+|S_{\eta\eta}(\eta,\xi)|+|W_{\eta\eta}(\eta,\xi)|}{\xi^2} &\leq\bigg(\fr{1}{2}+\tilde{\delta}^2\bigg)\tilde{M} +\bigg(1+\fr{4}{3}\tilde{\delta}\bigg)K.
\end{align}
By choosing $\tilde{M}=10K$ and $\tilde{\delta}=\min\{\fr{1}{2},\sqrt[3]{\fr{3(x_2-x_1)}{4}}\}$, we observe
$$
\bigg(\fr{1}{2}+\tilde{\delta}^2\bigg)\tilde{M} +\bigg(1+\fr{4}{3}\tilde{\delta}\bigg)K\leq\fr{11}{12}\tilde{M}<\tilde{M}.
$$
We combine with \eqref{s1.16} and \eqref{s1.17} to see that Properties (P2)-(P4) are preserved by the mapping $\widetilde{\mathcal{T}}$ for $\tilde{\delta}, \tilde{M}$ chosen above. To determine $\widetilde{\mathcal{T}}(\mathbf{F})\in\mathcal{S}^{\tilde{M}}_{\tilde{\delta}}$, we first find by \eqref{s1.12} that $R(\eta,0)=S(\eta, 0)=W(\eta, 0)=0$. Then it suffices to show that $R_\xi(\eta,0)=S_\xi(\eta, 0)=W_\xi(\eta, 0)=0$.
Differentiating system \eqref{s1.12} with respect to $\xi$ leads to
\begin{align}\label{s1.18}
\fr{\pa R}{\pa \xi}(\eta,\xi)&=\fr{r(\eta,\xi)-s(\eta,\xi)}{2\xi}-2\xi^2(u_1'(\eta)+u_0''(\eta)\xi) \nonumber \\[4pt]
&\quad -2\xi^2 \displaystyle\int_{0}^\xi\bigg(\fr{r_x(x_+(t),t)-s_x(x_+(t),t)}{2t}-2t^2(u_1''(x_+(t))+u_0'''(x_+(t))t)\bigg)\ {\rm d}t,\nonumber \\[4pt]
\fr{\pa S}{\pa \xi}(\eta,\xi)&=\fr{s(\eta,\xi)-r(\eta,\xi)}{2\xi}+2\xi^2(u_1'(\eta)-u_0''(\eta)\xi)  \\[4pt]
&\quad +2\xi^2 \displaystyle\int_{0}^\xi\bigg(\fr{s_x(x_-(t),t)-r_x(x_-(t),t)}{2t}+2t^2(u_1''(x_-(t))-u_0'''(x_-(t))t)\bigg)\ {\rm d}t, \nonumber \\[4pt]
\fr{\pa W}{\pa \xi}(\eta,\xi)&=-2\xi\{r(\eta,\xi)+u_1(\eta)\} -4\xi^2\displaystyle\int_{0}^\xi t\{r_x(x_-(t),t)+u_1'(x_-(t))\}\ {\rm d}t. \nonumber
\end{align}
It is obvious by \eqref{s1.15} and \eqref{s1.18} that $R_\xi(\eta,0)=S_\xi(\eta, 0)=W_\xi(\eta, 0)=0$, which indicates that $\widetilde{\mathcal{T}}$ do map $\mathcal{S}^{\tilde{M}}_{\tilde{\delta}}$ onto itself.

We now establish \eqref{s1.14} for $\tilde{\nu}=\fr{3}{4}$. For $(\hat{R},\hat{S},\hat{W})$, it follows that
\begin{align}\label{s1.19}
\left\{
\begin{array}{l}
\fr{\rm d}{{\rm d}_+ t}\hat{R}=\fr{\hat{r}-\hat{s}}{2t}-2t^2(u_1'+u_0''t),\\[4pt]
\fr{\rm d}{{\rm d}_- t}\hat{S}=\fr{\hat{s}-\hat{r}}{2t}+2t^2(u_1'-u_0''t), \\[4pt]
\fr{\rm d}{{\rm d}_- t}\hat{W}=-2t(\hat{r}+u_1).
\end{array}
\right.
\end{align}
Subtracting \eqref{s1.19} from \eqref{s1.11} gives
\begin{align}\label{s1.20}
\left\{
\begin{array}{l}
\fr{\rm d}{{\rm d}_+ t}(R-\hat{R})=\fr{(r-\hat{r})-(s-\hat{s})}{2t},\\[4pt]
\fr{\rm d}{{\rm d}_- t}(S-\hat{S})=\fr{(s-\hat{s})-(r-\hat{r})}{2t}, \\[4pt]
\fr{\rm d}{{\rm d}_- t}(W-\hat{W})=-2t(r-\hat{r}),
\end{array}
\right.
\end{align}
from which one arrives at
\begin{align*}
|R-\hat{R}|&\leq\int_{0}^\xi\fr{|r-\hat{r}|+|s-\hat{s}|}{2t}\ {\rm d}t\leq\fr{1}{4}\xi^2d(\mathbf{F},\hat{\mathbf{F}}),\\[4pt]
|S-\hat{S}|&\leq\int_{0}^\xi\fr{|r-\hat{r}|+|s-\hat{s}|}{2t}\ {\rm d}t\leq\fr{1}{4}\xi^2d(\mathbf{F},\hat{\mathbf{F}}), \\[4pt]
|W-\hat{W}|&\leq\int_{0}^\xi2t|r-\hat{r}|\ {\rm d}t\leq \xi^4d(\mathbf{F},\hat{\mathbf{F}}).
\end{align*}
Then we obtain
\begin{align*}
d\bigg(\widetilde{\mathcal{T}}(\mathbf{F}),\widetilde{\mathcal{T}}(\widehat{\mathbf{F}})\bigg) &=\bigg\|\fr{R-\hat{R}}{\xi^2}\bigg\|_{\infty} +\bigg\|\fr{S-\hat{S}}{\xi^2}\bigg\|_{\infty} +\bigg\|\fr{W-\hat{W}}{\xi^2}\bigg\|_{\infty} \\[4pt]
&\leq \bigg(\fr{1}{2}+\tilde{\delta}^2\bigg)d(\mathbf{F},\hat{\mathbf{F}}) \leq\fr{3}{4}d(\mathbf{F},\hat{\mathbf{F}}),
\end{align*}
which means that $\widetilde{\mathcal{T}}$ is a contraction under the metric $d$.
\end{proof}

\noindent {\bf Step 3 (Properties of the limit function).}
Since $(\mathcal{S}^{\tilde{M}}_{\tilde{\delta}},d)$ is not a closed subset in the complete space $(\mathcal{W}^{\tilde{M}}_{\tilde{\delta}},d)$, we need to confirm that the limit vector function of the iteration sequence $\{\mathbf{F}^{(n)}\}$, defined by $\mathbf{F}^{(n)}=\widetilde{\mathcal{T}}\mathbf{F}^{(n-1)}$, is in $\mathcal{S}^{\tilde{M}}_{\tilde{\delta}}$. This follows directly from Arzela-
Ascoli Theorem and the following lemma.
\begin{lem}\label{lemT2}
With the assumptions in Theorem \ref{thmT1}, the iteration sequence $\{\mathbf{F}^{(n)}\}$ has the property that $\{\pa_t \mathbf{F}^{(n)}(x,t)\}$ and $\{\pa_x\mathbf{F}^{(n)}(x,t)\}$ are uniformly Lipschitz continuous on $\widetilde{D}_{\tilde{\delta}}$.
\end{lem}
\begin{proof}
Let $(r,s,w)^T\in\mathcal{S}^{\tilde{M}}_{\tilde{\delta}}$. Then we obtain by Lemma \ref{lemT1} that
$(R,S,W)^T=\widetilde{\mathcal{T}}(r,s,w)^T$ also in $\mathcal{S}^{\tilde{M}}_{\tilde{\delta}}$. We divide the proof of this lemma into three steps.

We first prove $|R_t|+|S_t|+|W_t|\leq 3\tilde{M}t$. It suggests by the first equation of \eqref{s1.18} that
\begin{align*}
\bigg|\fr{\pa R}{\pa \xi}(\eta,\xi)\bigg|&\leq\fr{\tilde{M}}{2}\xi+2K\xi^2 +2\xi^2 \displaystyle\int_{0}^\xi\bigg(\fr{\tilde{M}}{2}t+2Kt^2\bigg)\ {\rm d}t \\[4pt]
& \leq\bigg(\fr{1}{2}+\fr{\tilde{\delta}}{5}+\fr{\tilde{\delta}^3}{2}+\fr{\tilde{\delta}^4}{15}\bigg)\tilde{M}\xi \leq\fr{7}{10}\tilde{M}\xi.
\end{align*}
Similarly one has $|R_\xi|\leq7\tilde{M}\xi/10$. Moreover, for $|W_\xi|$, we have by the third equation of \eqref{s1.18}
\begin{align*}
\bigg|\fr{\pa W}{\pa \xi}(\eta,\xi)\bigg|&\leq2\xi\{\tilde{M}\tilde{\delta}^2+K\} +4\xi^2\displaystyle\int_{0}^\xi t\{\tilde{M}t^2+K\}\ {\rm d}t \\[4pt]
&\leq \bigg(2\tilde{\delta}^2+\fr{1}{5}+2\tilde{\delta}^5+\fr{\tilde{\delta}^3}{5}\bigg)\tilde{M}\xi \leq\fr{9}{10}\tilde{M}\xi.
\end{align*}
Thus we get
\begin{align}\label{s1.21}
 \bigg|\fr{\pa R}{\pa \xi}(\eta,\xi)\bigg|+\bigg|\fr{\pa S}{\pa \xi}(\eta,\xi)\bigg| +\bigg|\fr{\pa W}{\pa \xi}(\eta,\xi)\bigg|\leq 3\tilde{M}\xi.
\end{align}

We next prove $|R_{tx}|+|S_{tx}|+|W_{tx}|\leq 3\tilde{M}t$. Differentiating system \eqref{s1.18} with respect to $\eta$ yields
\begin{align*}
&\fr{\pa^2 R}{\pa\eta\pa \xi}(\eta,\xi)=\fr{r_\eta(\eta,\xi) -s_\eta(\eta,\xi)}{2\xi}-2\xi^2(u_1''(\eta)+u_0'''(\eta)\xi) \nonumber \\[4pt]
&\quad  -2\xi^2 \displaystyle\int_{0}^\xi\bigg(\fr{r_{xx}(x_+(t),t)-s_{xx}(x_+(t),t)}{2t}-2t^2(u_1'''(x_+(t))+u_0''''(x_+(t))t)\bigg)\ {\rm d}t,\nonumber \\[4pt]
&\fr{\pa^2 S}{\pa\eta\pa \xi}(\eta,\xi) =\fr{s_\eta(\eta,\xi)-r_\eta(\eta,\xi)}{2\xi}+2\xi^2(u_1''(\eta)-u_0'''(\eta)\xi)  \\[4pt]
&\quad +2\xi^2 \displaystyle\int_{0}^\xi\bigg(\fr{s_{xx}(x_-(t),t)-r_{xx}(x_-(t),t)}{2t}+2t^2(u_1'''(x_-(t))-u_0''''(x_-(t))t)\bigg)\ {\rm d}t, \nonumber \\[4pt]
&\fr{\pa^2 W}{\pa\eta\pa \xi}(\eta,\xi)=-2\xi\{r_\eta(\eta,\xi)+u_1'(\eta)\} -4\xi^2\displaystyle\int_{0}^\xi t\{r_{xx}(x_-(t),t)+u_1''(x_-(t))\}\ {\rm d}t, \nonumber
\end{align*}
from which we obtain
\begin{align}\label{s1.22}
 \bigg|\fr{\pa^2 R}{\pa\eta\pa \xi}(\eta,\xi)\bigg|+\bigg|\fr{\pa^2 S}{\pa\eta\pa \xi}(\eta,\xi)\bigg| +\bigg|\fr{\pa^2 W}{\pa\eta\pa \xi}(\eta,\xi)\bigg|\leq 3\tilde{M}\xi.
\end{align}

Finally, we show $|R_{tt}|+|S_{tt}|+|W_{tt}|\leq 13\tilde{M}$. We differentiate system \eqref{s1.18} with respect to $\xi$ to achieve
\begin{align}
\fr{\pa^2 R}{\pa \xi^2}(\eta,\xi)&=\fr{r_\xi-s_\xi}{\xi}-\fr{r-s}{\xi^2}-4\xi(2u_1'(\eta)+3u_0''(\eta)\xi) \nonumber \\[4pt]
&\  -4\xi \displaystyle\int_{0}^\xi\bigg(\fr{r_x(x_+(t),t)-s_x(x_+(t),t)}{2t}-2t^2(u_1''(x_+(t)) +u_0'''(x_+(t))t)\bigg)\ {\rm d}t \nonumber \\[4pt]
&\  +4\xi^4\displaystyle\int_{0}^\xi\bigg(\fr{r_{xx}(x_+(t),t)-s_{xx}(x_+(t),t)}{2t}-2t^2(u_1'''(x_+(t)) +u_0''''(x_+(t))t)\bigg)\ {\rm d}t, \nonumber
\end{align}
\begin{align*}
\fr{\pa^2 S}{\pa \xi^2}(\eta,\xi)&=\fr{s_\xi-r_\xi}{\xi} -\fr{s-r}{\xi^2} +4\xi(2u_1'-3u_0''\xi)  \\[4pt]
&\  +4\xi \displaystyle\int_{0}^\xi\bigg(\fr{s_x(x_-(t),t)-r_x(x_-(t),t)}{2t}+2t^2(u_1''(x_-(t))-u_0'''(x_-(t))t)\bigg)\ {\rm d}t \\[4pt]
&\ +4\xi^4 \displaystyle\int_{0}^\xi\bigg(\fr{s_{xx}(x_-(t),t)-r_{xx}(x_-(t),t)}{2t}+2t^2(u_1'''(x_-(t))-u_0''''(x_-(t))t)\bigg)\ {\rm d}t
\end{align*}
\begin{align*}
\fr{\pa^2 W}{\pa \xi^2}(\eta,\xi)=&-4\{r+u_1+\xi r_\xi\} -8\xi\displaystyle\int_{0}^\xi t\{r_x(x_-(t),t)+u_1'(x_-(t))\}\ {\rm d}t \\[4pt]
&-8\xi^4\int_{0}^\xi t\{r_{xx}(x_-(t),t)+u_1''(x_-(t))\}\ {\rm d}t
\end{align*}
which together with \eqref{s1.15} and \eqref{s1.21} gives
\begin{align*}
\bigg|\fr{\pa^2 R}{\pa \xi^2}(\eta,\xi)\bigg|&\leq 3\tilde{M}+\tilde{M}+4\xi K+4\xi(1+\xi^3)\int_{0}^\xi\bigg(\fr{\tilde{M}}{2}t+2Kt^2\bigg)\ {\rm d}t\leq 5\tilde{M}, \\[4pt]
\bigg|\fr{\pa^2 S}{\pa \xi^2}(\eta,\xi)\bigg|&\leq 3\tilde{M}+\tilde{M}+4\xi K+4\xi(1+\xi^3)\int_{0}^\xi\bigg(\fr{\tilde{M}}{2}t+2Kt^2\bigg)\ {\rm d}t\leq 5\tilde{M}, \\[4pt]
\bigg|\fr{\pa^2 S}{\pa \xi^2}(\eta,\xi)\bigg|&\leq 4(\tilde{M}\xi^2+K+\tilde{M}\xi^3)+8\xi(1+\xi^3)\int_{0}^\xi t(\tilde{M}t^2+K)\ {\rm d}t\leq3\tilde{M}.
\end{align*}
Thus we have
\begin{align}\label{s1.23}
\bigg|\fr{\pa^2 R}{\pa \xi^2}(\eta,\xi)\bigg|+\bigg|\fr{\pa^2 S}{\pa \xi^2}(\eta,\xi)\bigg| +\bigg|\fr{\pa^2 W}{\pa \xi^2}(\eta,\xi)\bigg|\leq 13\tilde{M}.
\end{align}
Combining \eqref{s1.21}-\eqref{s1.23} ends the proof of the lemma.
\end{proof}
Hence, the proof of Theorem \ref{thmT1} is completed.

\section{The derivation of formulations}

\subsection{The derivation of \eqref{2.7}} \label{App1}

We only derive the third equation of \eqref{2.7} since the derivation of the fourth equation is parallel. The first two equations can be derived obviously. Putting the expression of $(u,v)$ in \eqref{2.4} into the third equation of \eqref{2.2} and using the notation of $\bar{\pa}^+$ in \eqref{2.5} yields
\begin{align}\label{A1}
-\fr{\sin\theta}{\cos\omega}\bar{\pa}^+\bigg(\fr{\cos\theta}{\sin\omega}\bigg)+ \fr{\cos\theta}{\cos\omega}\bar{\pa}^+\bigg(\fr{\sin\theta}{\sin\omega}\bigg)+\fr{1}{\gamma p}\bar{\pa}^+p=0.
\end{align}
On the other hand, we have by the Bernoulli function
\begin{align*}
B=\fr{u^2+v^2}{2}+\fr{c^2}{\gamma-1}&=\bigg(\fr{1}{2\sin^2\omega}+\fr{1}{\gamma-1}\bigg)c^2 \\[4pt]
&=\fr{\gamma(\kappa+\sin^2\omega)}{2\kappa\sin^2\omega}\cdot\fr{p}{\rho},
\end{align*}
which, together with the entropy function $S=p\rho^{-\gamma}$, gives
\begin{align*}
\ln p=\fr{\gamma}{\gamma-1}\bigg(\ln B-\fr{1}{\gamma}\ln S-\ln\fr{\gamma(\kappa+\sin^2\omega)}{2\kappa\sin^2\omega}\bigg),
\end{align*}
and further provides
\begin{align*}
\fr{1}{p}\bar{\pa}^+p=\fr{\gamma}{2\kappa}\bar{\pa}^+(\ln B-\fr{1}{\gamma}\ln S)+\fr{\gamma\cot\omega}{\kappa+\sin^2\omega}\bar{\pa}^+\omega.
\end{align*}
Inserting the above into \eqref{A1} leads to
\begin{align*}
&-\fr{\sin\theta}{\cos\omega} \cdot\fr{-\sin\theta\sin\omega\bar{\pa}^+\theta-\cos\theta\cos\omega\bar{\pa}^+\omega}{\sin^2\omega} +\fr{\cos\theta}{\cos\omega} \cdot\fr{\cos\theta\sin\omega\bar{\pa}^+\theta-\sin\theta\cos\omega\bar{\pa}^+\omega}{\sin^2\omega} \\[4pt] &\qquad +\fr{1}{2\kappa}\bar{\pa}^+\bigg(\ln B-\fr{1}{\gamma}\ln S\bigg)+\fr{\cot\omega}{\kappa+\sin^2\omega}\bar{\pa}^+\omega=0.
\end{align*}
Doing a simple rearrangement arrives at
$$
\bar{\pa}^+\theta+\fr{\cos^2\omega}{\kappa+\sin^2\omega}\bar{\pa}^+\omega +\fr{\sin(2\omega)}{4\kappa}\bar{\pa}^+\bigg(\ln B-\fr{1}{\gamma}\ln S\bigg)=0,
$$
which is the third equation of \eqref{2.7}.

\subsection{The commutator relation between $\bar{\pa}^0$ and $\bar{\pa}^+$}\label{App2}

By a direct calculation, we obtain from \eqref{2.5},
\begin{align}\label{B1}
\bar{\pa}^0\bar{\pa}^+&=(\cos\theta\pa_x+\sin\theta\pa_y)(\cos\alpha\pa_x+\sin\alpha\pa_y) \nonumber \\ &=\cos\theta\cos\alpha\pa_{xx}+(\cos\theta\sin\alpha+\sin\theta\cos\alpha)\pa_{xy} +\sin\theta\sin\alpha\pa_{yy} \nonumber \\
&\quad -\sin\alpha(\cos\theta\alpha_x+\sin\theta\alpha_y)\pa_x +\cos\alpha(\cos\theta\alpha_x+\sin\theta\alpha_y)\pa_y
\nonumber \\ &= \cos\theta\cos\alpha\pa_{xx}+(\cos\theta\sin\alpha+\sin\theta\cos\alpha)\pa_{xy} +\sin\theta\sin\alpha\pa_{yy} \nonumber \\
&\quad -\sin\alpha\bar{\pa}^0\alpha\pa_x +\cos\alpha\bar{\pa}^0\alpha\pa_y,
\end{align}
and
\begin{align}\label{B2}
\bar{\pa}^+\bar{\pa}^0&=(\cos\alpha\pa_x+\sin\alpha\pa_y)(\cos\theta\pa_x+\sin\theta\pa_y) \nonumber \\ &=\cos\theta\cos\alpha\pa_{xx}+(\cos\theta\sin\alpha+\sin\theta\cos\alpha)\pa_{xy} +\sin\theta\sin\alpha\pa_{yy} \nonumber \\
&\quad -\sin\theta(\cos\alpha\theta_x+\sin\alpha\theta_y)\pa_x +\cos\theta(\cos\alpha\theta_x+\sin\alpha\theta_y)\pa_y
\nonumber \\ &= \cos\theta\cos\alpha\pa_{xx}+(\cos\theta\sin\alpha+\sin\theta\cos\alpha)\pa_{xy} +\sin\theta\sin\alpha\pa_{yy} \nonumber \\
&\quad -\sin\theta\bar{\pa}^+\theta\pa_x +\cos\theta\bar{\pa}^+\theta\pa_y.
\end{align}
Subtracting \eqref{B2} from \eqref{B1} yields
\begin{align}\label{B3}
\bar{\pa}^0\bar{\pa}^+-\bar{\pa}^+\bar{\pa}^0 =(\sin\theta\bar{\pa}^+\theta-\sin\alpha\bar{\pa}^0\alpha)\pa_x -(\cos\theta\bar{\pa}^+\theta-\cos\alpha\bar{\pa}^0\alpha)\pa_y.
\end{align}
On the other hand, we find from \eqref{2.6} and \eqref{2.4},
\begin{align*}
\pa_x&=\fr{2\sin\alpha\cos\omega\bar{\pa}^0-(\sin\alpha+\sin\beta)\bar{\pa}^+}{\sin(2\omega)} =\fr{\sin\alpha\bar{\pa}^0-\sin\theta\bar{\pa}^+}{\sin\omega}, \\[4pt] \pa_y&=-\fr{2\cos\alpha\cos\omega\bar{\pa}^0-(\cos\alpha+\cos\beta)\bar{\pa}^+}{\sin(2\omega)} =-\fr{\cos\alpha\bar{\pa}^0-\cos\theta\bar{\pa}^+}{\sin\omega}.
\end{align*}
Inserting the above into \eqref{B3}, one has
\begin{align*}
&\bar{\pa}^0\bar{\pa}^+-\bar{\pa}^+\bar{\pa}^0 \\[4pt]
=&\fr{(\sin\alpha\sin\theta+\cos\alpha\cos\theta)\bar{\pa}^+\theta -\bar{\pa}^0\alpha}{\sin\omega}\bar{\pa}^0 -\fr{\bar{\pa}^+\theta-(\sin\theta\sin\alpha+\cos\theta\cos\alpha)\bar{\pa}^0\alpha}{\sin\omega}\bar{\pa}^+ \\[4pt]
=& \fr{\cos\omega\bar{\pa}^+\theta -\bar{\pa}^0\alpha}{\sin\omega}\bar{\pa}^0 -\fr{\bar{\pa}^+\theta-\cos\omega\bar{\pa}^0\alpha}{\sin\omega}\bar{\pa}^+,
\end{align*}
which leads to the desired result.

\subsection{The characteristic decomposition for $\Xi$}\label{App3}

We only derive the first one of \eqref{2.16}, and the other can be obtained with the same technique. Firstly, we have the following commutator relation between $\bar{\pa}^-$ and $\bar{\pa}^+$
\begin{align*}
\bar{\pa}^-\bar{\pa}^+-\bar{\pa}^+\bar{\pa}^- =\fr{\cos(2\omega)\bar{\pa}^+\beta-\bar{\pa}^-\alpha}{\sin(2\omega)}\bar{\pa}^- -\fr{\bar{\pa}^+\beta-\cos(2\omega)\bar{\pa}^-\alpha}{\sin(2\omega)}\bar{\pa}^+,
\end{align*}
which can be derived in a similar way as in Appendix B or see Li and Zheng \cite{Li-Zheng1}. Thus it presents from \eqref{2.15}  that
\begin{align}\label{C1}
\bar{\pa}^-\bar{\pa}^+\theta-\bar{\pa}^+\bar{\pa}^-\theta &=\fr{\cos(2\omega)\bar{\pa}^+\beta-\bar{\pa}^-\alpha}{\sin(2\omega)}\bar{\pa}^-\theta -\fr{\bar{\pa}^+\beta-\cos(2\omega)\bar{\pa}^-\alpha}{\sin(2\omega)}\bar{\pa}^+\theta \nonumber \\[4pt]
&=[\cos(2\omega)\bar{\pa}^+\beta-\bar{\pa}^-\alpha]\bar{\pa}^-\Xi +[\bar{\pa}^+\beta-\cos(2\omega)\bar{\pa}^-\alpha]\bar{\pa}^+\Xi,
\end{align}
and
\begin{align}\label{C2}
\bar{\pa}^+\bar{\pa}^-\Xi =\bar{\pa}^-\bar{\pa}^+\Xi-\fr{\cos(2\omega)\bar{\pa}^+\beta-\bar{\pa}^-\alpha}{\sin(2\omega)}\bar{\pa}^-\Xi +\fr{\bar{\pa}^+\beta-\cos(2\omega)\bar{\pa}^-\alpha}{\sin(2\omega)}\bar{\pa}^+\Xi.
\end{align}
On the other hand, we find by using \eqref{2.15} again that
\begin{align*}
\bar{\pa}^-\bar{\pa}^+\theta&=-\sin(2\omega)\bar{\pa}^-\bar{\pa}^+\Xi-2\cos(2\omega)\bar{\pa}^-\omega\bar{\pa}^+\Xi, \\
\bar{\pa}^+\bar{\pa}^-\theta&=\sin(2\omega)\bar{\pa}^+\bar{\pa}^-\Xi+2\cos(2\omega)\bar{\pa}^+\omega\bar{\pa}^-\Xi.
\end{align*}
Substituting the above into \eqref{C1} and rearranging the result achieve
\begin{align}\label{C3}
-\sin(2\omega)(\bar{\pa}^-\bar{\pa}^+\Xi+\bar{\pa}^+\bar{\pa}^-\Xi)=&\bigg(2\cos(2\omega)\bar{\pa}^+\omega +\cos(2\omega)\bar{\pa}^+\beta-\bar{\pa}^-\alpha\bigg)\bar{\pa}^-\Xi \nonumber \\
&+\bigg(2\cos(2\omega)\bar{\pa}^-\omega+\bar{\pa}^+\beta -\cos(2\omega)\bar{\pa}^-\alpha\bigg)\bar{\pa}^+\Xi.
\end{align}
We put \eqref{C2} into \eqref{C3} to obtain
\begin{align}\label{C4}
\bar{\pa}^-\bar{\pa}^+\Xi&=-\fr{\cos(2\omega)\bar{\pa}^+\omega}{\sin(2\omega)}\bar{\pa}^-\Xi -\fr{\cos(2\omega)\bar{\pa}^-\omega+\bar{\pa}^+\beta -\cos(2\omega)\bar{\pa}^-\alpha}{\sin(2\omega)}\bar{\pa}^+\Xi \nonumber \\[4pt]
&=-\fr{\cos(2\omega)\bar{\pa}^+\omega}{\sin(2\omega)}\bar{\pa}^-\Xi -\fr{\bar{\pa}^+\theta-\bar{\pa}^+\omega-\cos(2\omega)\bar{\pa}^-\theta}{\sin(2\omega)}\bar{\pa}^+\Xi
\nonumber \\[4pt]
&=-\fr{\cos(2\omega)\bar{\pa}^+\omega}{\sin(2\omega)}\bar{\pa}^-\Xi +\fr{\sin(2\omega)\bar{\pa}^+\Xi +\bar{\pa}^+\omega+\sin(2\omega)\cos(2\omega)\bar{\pa}^-\Xi}{\sin(2\omega)}\bar{\pa}^+\Xi.
\end{align}
Moreover, thanks to \eqref{2.12} and \eqref{2.15}, one arrives at
\begin{align*}
\bar{\pa}^+\omega=\fr{2\sin\omega(\kappa+\sin^2\omega)}{\cos\omega}(\bar{\pa}^+\Xi+GH).
\end{align*}
It follows by inserting the above into \eqref{C4} that
\begin{align*}
&\bar{\pa}^-\bar{\pa}^+\Xi=-\fr{\cos(2\omega)(\kappa+\sin^2\omega)}{\cos^2\omega}(\bar{\pa}^+\Xi+GH)\bar{\pa}^-\Xi \\[4pt] &\qquad \qquad \qquad  +\fr{\kappa+\sin^2\omega}{\cos^2\omega}(\bar{\pa}^+\Xi+GH)\bar{\pa}^+\Xi +[\bar{\pa}^+\Xi+\cos(2\omega)\bar{\pa}^-\Xi]\bar{\pa}^+\Xi  \\[4pt]
=&\fr{\kappa+\sin^2\omega}{\cos^2\omega}(\bar{\pa}^+\Xi+GH)[\bar{\pa}^+\Xi-\cos(2\omega)\bar{\pa}^-\Xi] +[\bar{\pa}^+\Xi+\cos(2\omega)\bar{\pa}^-\Xi]\bar{\pa}^+\Xi \\[4pt]
=&\fr{\kappa\bar{\pa}^+\Xi+(\kappa+\sin^2\omega)GH}{\cos^2\omega}[\bar{\pa}^+\Xi-\cos(2\omega)\bar{\pa}^-\Xi] +\fr{\bar{\pa}^+\Xi}{\cos^2\omega}[\bar{\pa}^+\Xi+\cos^2(2\omega)\bar{\pa}^-\Xi],
\end{align*}
which is the first equation of \eqref{2.16}.


\end{document}